\documentclass{amsart}

\usepackage[utf8x]{inputenc}
\usepackage[T1]{fontenc}

\usepackage{amsmath,amsfonts,amssymb,amsthm}

\usepackage{hyperref}

\usepackage{booktabs}

\usepackage{float}

\usepackage[shortlabels]{enumitem}

\usepackage{multicol}

\usepackage{verbatim}

\usepackage[linesnumbered,ruled,vlined]{algorithm2e}
    \DontPrintSemicolon
    \SetAlgoSkip{medskip}

\theoremstyle{plain}   \newtheorem{theorem}{Theorem}[section]
\theoremstyle{plain}   \newtheorem{proposition}[theorem]{Proposition}
\theoremstyle{plain}   \newtheorem{corollary}[theorem]{Corollary}
\theoremstyle{plain}   \newtheorem{lemma}[theorem]{Lemma}
\theoremstyle{definition}   
\theoremstyle{definition}   \newtheorem{example}[theorem]{Example}

\newcommand{\nset}{\mathbb{N}}
\newcommand{\aset}{\mathcal{A}}
\newcommand{\anset}{\mathcal{A}_n}
\newcommand{\anmset}{\mathcal{A}_{n-1}}

\newcommand*{\psymb}{\mathrm{P}}

\newcommand*{\plac}{{\mathsf{plac}}}

\newcommand*{\hypo}{{\mathsf{hypo}}}
\newcommand*{\hypon}{{\mathsf{hypo}_n}}

\newcommand*{\sylv}{{\mathsf{sylv}}}
\newcommand*{\sylvn}{{\mathsf{sylv}_n}}
\newcommand*{\sylvnm}{{\mathsf{sylv}_{n-1}}}
\newcommand*{\sylvnp}{{\mathsf{sylv}_{n+1}}}

\newcommand*{\sylvh}{{\mathsf{sylv}^{\#}}}
\newcommand*{\sylvhn}{{\mathsf{sylv}^{\#}_n}}

\newcommand*{\baxt}{{\mathsf{baxt}}}
\newcommand*{\baxtn}{{\mathsf{baxt}_n}}

\newcommand\chyph{\penalty 10000-\hskip 0pt\relax}

\usepackage{tikz}
\usetikzlibrary{arrows.meta,calc,decorations.pathmorphing,decorations.text,matrix,positioning}
\usetikzlibrary{shapes.geometric,shapes.misc,shapes.symbols}

\usetikzlibrary{matrix}

\tikzset{
  pretableaumatrix/.style={
    ampersand replacement=\&,
    matrix of math nodes,
    outer sep=1mm,
    inner sep=0mm,
    anchor=center,
    row sep={between borders,-\pgflinewidth},
    column sep={between borders,-\pgflinewidth},
    dottedentry/.style={densely dotted},
    spaceentry/.style={draw=none,execute at begin node=\null},
  },
  pretableaunode/.style={
    font=\small,
    draw=gray,
    sharp corners,
    rectangle,
    anchor=base,
    text height=3.75mm,
    text depth=1.25mm,
    minimum height=5mm,
    minimum width=5mm,
    inner sep=0mm,
    outer sep=0mm,
  },
  tableaumatrix/.style={
    pretableaumatrix,
    every node/.append style={
      pretableaunode,
    },
  },
  medtableaumatrix/.style={
    pretableaumatrix,
    every node/.append style={
      pretableaunode,
      font=\footnotesize,
      text height=2.75mm,
      text depth=.75mm,
      minimum height=3.5mm,
      minimum width=3.5mm
    },
  },
  smalltableaumatrix/.style={
    pretableaumatrix,
    every node/.append style={
      pretableaunode,
      font=\scriptsize,
      text height=1.85mm,
      text depth=.15mm,
      minimum height=2.5mm,
      minimum width=2.5mm,
    },
  },
  tinytableaumatrix/.style={
    pretableaumatrix,
    every node/.append style={
      pretableaunode,
      font=\tiny,
      text height=1.25mm,
      text depth=.15mm,
      minimum height=1.75mm,
      minimum width=1.75mm
    },
  },
  tableau/.style={
    baseline=-1.25mm,
    every matrix/.style={tableaumatrix},
  },
  medtableau/.style={
    baseline=-1.25mm,
    every matrix/.style={medtableaumatrix},
  },
  smalltableau/.style={
    baseline=-1.25mm,
    every matrix/.style={smalltableaumatrix},
  },
  preshapetableaumatrix/.style={
    pretableaumatrix,
    execute at end cell={\strut},
    every node/.append style={
      draw=black,
      anchor=base,
      inner sep=0mm,
      outer sep=0mm,
    },
    shadedentry/.style={fill=gray},
    darkshadedentry/.style={fill=darkgray},
  },
  medshapetableaumatrix/.style={
    preshapetableaumatrix,
    every node/.append style={
      text height=2.75mm,
      text depth=.75mm,
      minimum height=3.5mm,
      minimum width=3.5mm
    },
  },
  shapetableaumatrix/.style={
    ampersand replacement=\&,
    matrix of math nodes,
    outer sep=0mm,
    inner sep=0mm,
    anchor=base,
    row sep={between borders,-\pgflinewidth},
    column sep={between borders,-\pgflinewidth},
    execute at begin cell={\strut},
    every node/.append style={draw,anchor=base,text height=1mm,text depth=.5mm,minimum size=1.5mm,inner sep=0mm,outer sep=0mm},
  },
  shapetableau/.style={
    every matrix/.style={shapetableaumatrix},
  },
  topalign/.style={
    every matrix/.append style={name=maintableau,anchor=maintableau-1-1.base},
    baseline,
  },
}

\usetikzlibrary{shapes.geometric} 
\usetikzlibrary{shapes.misc}  

\tikzset{
  bst/.style={
    standard/.style={
      font=\small,
      draw=gray,
      rounded rectangle,
      minimum width=4.5mm,
      minimum height=4.5mm,
      inner xsep=0mm,
      inner ysep=1mm,
      outer sep=0mm,
      line width=.5pt,
    },
    empty/.style={
      minimum width=3mm,
      minimum height=3mm,
    },
    triangle/.style={
      isosceles triangle,
      isosceles triangle apex angle=60,
      shape border rotate=90,
      rounded corners=2mm,
      minimum width=8mm,
      inner xsep=0mm,
      inner ysep=.5mm
    },
    blank/.style={
      draw=none,
    },
    nodecount/.style={
      blank,
      font=\scriptsize,
    },
    every node/.style={standard},
    every child/.style={draw=black,line width=.6pt},
    level distance=10mm,
    level 1/.style={sibling distance=60mm},
    level 2/.style={sibling distance=30mm},
    level 3/.style={sibling distance=15mm},
  },
  medbst/.style={
    bst,
    level distance=10mm,
    level 1/.style={sibling distance=15mm},
    level 2/.style={sibling distance=15mm},
    level 3/.style={sibling distance=15mm},
  },
  smallbst/.style={
    bst,
    level distance=8mm,
    level 1/.style={sibling distance=10mm},
    level 2/.style={sibling distance=10mm},
    level 3/.style={sibling distance=10mm},
  },
  tinybst/.style={
    bst,
    level distance=5mm,
    level 1/.style={sibling distance=8mm},
    level 2/.style={sibling distance=8mm},
    level 3/.style={sibling distance=8mm},
    every node/.append style={
      font=\footnotesize,
    },
    triangle/.append style={
      rounded corners=1mm,
      minimum width=7mm,
      inner xsep=-.5mm,
    },
  },
  microbst/.style={
    bst,
    standard/.append style={
      font=\scriptsize,
      minimum width=3mm,
      minimum height=3mm,
      inner ysep=.25mm,
    },
    level distance=3mm,
    level 1/.style={sibling distance=6mm},
    level 2/.style={sibling distance=6mm},
    level 3/.style={sibling distance=6mm},
  },
  nanobst/.style={
    bst,
    standard/.append style={
      font=\tiny,
      minimum width=2mm,
      minimum height=2mm,
      inner ysep=.25mm,
    },
    level distance=2mm,
    level 1/.style={sibling distance=4mm},
    level 2/.style={sibling distance=4mm},
    level 3/.style={sibling distance=4mm},
  },
}
\begin{document}

    \title{Identities and bases in the sylvester and Baxter monoids}

    \author{Alan J. Cain}
    \address{
        Centro de Matem\'{a}tica e Aplica\c{c}\~{o}es\\
        Faculdade de Ci\^{e}ncias e Tecnologia\\
        Universidade Nova de Lisboa\\
        2829--516 Caparica\\
        Portugal
    }
    \email{a.cain@fct.unl.pt}

    \author{Ant\'{o}nio Malheiro}
    \address{
        Departamento de Matem\'{a}tica \& Centro de Matem\'{a}tica e Aplica\c{c}\~{o}es\\
        Faculdade de Ci\^{e}ncias e Tecnologia\\
        Universidade Nova de Lisboa\\
        2829--516 Caparica\\
        Portugal
    }
    \email{ajm@fct.unl.pt}

    \author{Duarte Ribeiro}
    \address{
        Departamento de Matem\'{a}tica \& Centro de Matem\'{a}tica e Aplica\c{c}\~{o}es\\
        Faculdade de Ci\^{e}ncias e Tecnologia\\
        Universidade Nova de Lisboa\\
        2829--516 Caparica\\
        Portugal
    }
    \email{dc.ribeiro@campus.fct.unl.pt}

    \thanks{
        This work is funded by National Funds through the FCT -- Funda\c{c}\~{a}o para a Ci\^{e}ncia e a Tecnologia, I.P., under the scope of the project {\scshape UIDB}/00297/2020 (Center for Mathematics and Applications) and the project {\scshape PTDC}/{\scshape MAT-PUR}/31174/2017.\\
        The third author is funded by National Funds through the FCT -- Funda\c{c}\~{a}o para a Ci\^{e}ncia e a Tecnologia, I.P., under the scope of the studentship {\scshape SFRH}/{\scshape BD}/138949/2018.
    }
    
    
    
    \begin{abstract}
        This paper presents new results on the identities satisfied by the sylvester and Baxter monoids. We show how to embed these monoids, of any rank strictly greater than $2$, into a direct product of copies of the corresponding monoid of rank $2$. This confirms that all monoids of the same family, of rank greater than or equal to $2$, satisfy exactly the same identities. We then give a complete characterization of those identities, and prove that the varieties generated by the sylvester and the Baxter monoids have finite axiomatic rank, by giving a finite basis for them. 
    \end{abstract}
    
	\maketitle

	
	\section{Introduction}
    \label{section:introduction}
	
	
	When studying the identities satisfied by a given semigroup $S$, if $S$ indeed satisfies a non-trivial identity, two natural questions arise: The first is the finite basis problem, that is, are the identities satisfied by $S$ consequences of those in some finite subset (see \cite{sapir_combinatorial,volkov_finitebasis}). There exist several powerful methods with which to approach the problem for finite semigroups, however, such is not the case for infinite semigroups. The second question is the computational complexity of the identity checking problem $\textsc{Check-Id}(S)$ \cite{kharlampovich_sapir_survey}, that is, the decision problem whose instance is an arbitrary identity $u \approx v$, and the answer to such an instance is `YES' if $S$ satisfies $u \approx v$, and `NO' if it does not. It is well-known that, for any finite semigroup $S$, the problem $\textsc{Check-Id}(S)$ is decidable, since there are only finitely many substitutions of the variables occurring in the identity by elements of $S$. Furthermore, $\textsc{Check-Id}(S)$ belongs in the complexity class \textbf{co-NP}. However, in the case of infinite semigroups, the brute-force approach used in the finite case does not work, and only recently there have been results on the computational complexity of identity checking for infinite semigroups, beyond undecidability and trivial or easy decidability in linear time (see \cite{daviaud_identities,volkov_kauffman3,volkov_kauffman4}).
	
    The plactic monoid $\plac$ \cite{LS1978}, whose elements can be identified with Young tableaux, has long been considered an important monoid, due to its numerous applications in different areas of mathematics, such as algebraic combinatorics \cite{lothaire_2002}, representation theory \cite{fulton_young,green2006polynomial}, symmetric functions \cite{Schutzenberger1977d,macdonald_symmetric}, Kostka-Foulkes polynomials \cite{lascoux_foulkes}, and crystal bases \cite{bump_crystalbases}. By its definition via Schensted's insertion algorithm \cite{Schensted1961}, the plactic monoid has decidable word problem. The question of identities satisfied by the plactic monoid is actively studied \cite{kubat_identities,izhakian_tropical}, since it is an infinite monoid with a powerful combinatorial structure. Cain \textit{et al} \cite{ckkmo_placticidentity} showed that the plactic monoid of finite rank $n$ does not satisfy any non-trivial identity of length less than or equal to $n$, thus showing that there is no single ``global'' identity satisfied by every plactic monoid of finite rank. On the other hand, Johnson and Kambites \cite{johnson_kambites_tropical_plactic} gave a tropical representation of the plactic monoid of every finite rank, thus showing that they all satisfy non-trivial identities. Furthermore, in conjunction with results given in \cite{daviaud_identities} and \cite{johnson_tran}, this shows that the identity checking problem in the plactic monoids of finite rank is in the complexity class $\mathbf{P}$. Daviaud \textit{et al} \cite{daviaud_identities} also show that the monoid of $2 \times 2$ upper-triangular tropical matrices (see, for example, \cite{maclagan_tropical}), the bicyclic monoid, and the plactic monoid of rank $2$ satisfy exactly the same identities. Since the bicyclic monoid is not finitely based \cite{shneerson_axiomaticrank}, none of these monoids are. 
    
    In the context of combinatorial Hopf algebras, whose bases are indexed by combinatorial objects, the plactic monoid is used to construct the Hopf algebra of free symmetric functions \textbf{FSym} \cite{poirier_reutenauer,duchamp2002noncommutative}, whose bases are indexed by standard Young tableaux. In this context, other monoids arise with similar combinatorial properties to the plactic monoid: the Hopf algebra \textbf{Sym} of non-commutative symmetric functions \cite{gelfandetal_noncommutative_symmetric_functions}, whose bases are indexed by integer compositions, is obtained from the hypoplactic monoid $\hypo$ \cite{Krob1997,novelli_hypoplactic}, the monoid of quasi-ribbon tableaux; the Loday–Ronco Hopf algebra \textbf{PBT} \cite{loday_ronco_hopf_planar_binary_trees,hivert_sylvester}, whose bases are indexed by planar binary trees, is obtained from the sylvester monoid $\sylv$ \cite{hivert_sylvester}, the monoid of right strict binary search trees; the Baxter Hopf algebra \textbf{Baxter} \cite{reading_lattice_hopf,giraudo_baxter2}, whose bases are indexed by Baxter permutations \cite{baxter_glen}, is obtained from the Baxter monoid $\baxt$ \cite{giraudo_baxter2}, the monoid of pairs of twin binary search trees. These monoids satisfy identities, and the shortest identities have been characterized \cite{cm_identities}. Unlike in the case of the plactic monoid, these identities are independent of rank, except for the case of rank $1$.

    The identities satisfied by the hypoplactic monoid have been studied in depth by the present authors in \cite{cmr_hypo_id}. It was shown that the hypoplactic monoids of rank greater than or equal to $2$ all satisfy exactly the same identities, which have been fully characterized. Furthermore, a finite basis was given for the variety generated by $\hypo$, thus proving that it has finite axiomatic rank. Although not stated in the article, the characterization of the identities implies that $\textsc{Check-Id}(\hypo)$ is in the complexity class $\mathbf{P}$. These results were obtained by extensively using an alternate characterization of the hypoplactic monoid using inversions, which arises as a consequence of \cite[Subsection~4.2]{novelli_hypoplactic}.
    
    This paper focuses on the sylvester and Baxter monoids, as well as the \#-sylvester monoid, whose elements are identified with left strict binary search trees and whose properties can be derived from those of the sylvester monoid by parallel reasoning. These monoids are closely related to each other  (see \cite[Proposition~3.7]{giraudo_baxter2}), as well as to the hypoplactic monoid (see \cite{priez_binary_trees,cm_sylv_crystal}). The main goal of the paper is to present a systematic study of the identities satisfied by these monoids, in the same way as the one given for the hypoplactic monoid in \cite{cmr_hypo_id}. This paper also gives an alternate characterization of these monoids, by introducing the concepts of right and left precedences, which serve the same purpose as inversions for the hypoplactic monoid. The authors of this paper remark that Theorem~\ref{theorem:sylv_finite_basis} has been proven independently in the upcoming paper \cite{kambites_johnson_cain_malheiro_preprint}.


	\section{Preliminaries and notation}
	\label{section:preliminaries}

	
	This section gives the necessary background on universal algebra (see \cite{bs_universal_algebra,mal2012algebraic,mckenzie2018algebras,Bergman_universal_algebra}), in the context of monoids, followed by the definition and essential facts about the sylvester and Baxter monoids.
	
	For the necessary background on semigroups and monoids, see \cite{howie1995fundamentals}; for presentations, see \cite{higgins1992techniques}; for computational complexity, see \cite{papadimitriou1994computational}; for a general background on the plactic monoid, see \cite[Chapter~5]{lothaire_2002}, and on the hypoplactic monoid, see \cite{novelli_hypoplactic} and \cite{cm_hypo_crystal}. 
	

    \subsection{Varieties, identities and bases}
    \label{subsection:varieties_identities_and_bases}
    
    The background given in this subsection is mostly identical to the one given in \cite[Subsection~2.1]{cmr_hypo_id}, of which this paper is a sequel. As such, we omit most of the subsection, with the following exceptions:
    
    We define the content and support of a balanced identity as the content and support of both sides of the identity, respectively.

    An equational theory $\Sigma$ is \textit{left 1-hereditary} if, for every identity $u \approx v$ of $\Sigma$ and any variable $x \in supp(u \approx v)$, the identity $u' \approx v'$ is in $\Sigma$, where $u_1$ (respectively, $v_1$) is the longest prefix of $u$ (respectively, $v$) where $x$ does not occur (see \cite{pastijn_polyhedral,volkov_finitebasis,mashevitzky_hereditary}). We define \textit{right 1-hereditary} equational theories in a dual way. The equational theory of the variety generated by the bicyclic monoid, which coincides with that of the variety generated by the plactic monoid of rank 2 (see \cite{johnson_kambites_tropical_plactic}), is both left and right 1-hereditary (see \cite{shneerson_axiomaticrank,pastijn_polyhedral}). 
    
    For a given semigroup $S$, its \textit{identity checking problem} $\textsc{Check-Id}(S)$ is the following combinatorial decision problem: the instance is an arbitrary identity $u \approx v$; the answer to such an instance is `YES', if $S$ satisfies the identity, and `NO' otherwise. Notice that the semigroup itself is fixed, as such, it is only the identity $u \approx v$ that serves as the input. Therefore, the time/space complexity of $\textsc{Check-Id}(S)$ should be measured only in terms of the size of the identity. Naturally, the problem can also be considered for monoids.
        
	
	\subsection{The sylvester and \#-sylvester monoids}
    \label{subsection:sylv_and_sylv_hash}
    
    This subsection gives a brief overview of the sylvester and \#-sylvester monoids and their related combinatorial objects and insertion algorithms, as well as results from \cite{cm_identities}. We introduce an alternative characterization of these monoids, analogous to the one given in \cite{novelli_hypoplactic} for the hypoplactic monoid, as well as some new notation. For more information on the sylvester monoid, see \cite{hivert_sylvester} and \cite{cm_sylv_crystal}; for more information on binary search trees, see \cite{knuth1970} and \cite{aho_ullman_foundations}.
    
    Let $\aset = \{ 1 < 2 < 3 < \cdots \}$ denote the set of positive integers, viewed as an infinite ordered alphabet, and let $\anset = \{ 1 < \cdots < n \}$ denote the set of the first $n$ positive integers, viewed as a finite ordered alphabet. For brevity, we will write `the node $a$' instead of `the node labelled with $a$'.

    A \textit{right strict binary search tree} is a labelled rooted binary tree where the label of each node is greater than or equal to the label of every node in its left subtree, and strictly less than every node in its right subtree. A \textit{left strict binary search tree} is a labelled rooted binary tree where the label of each node is strictly greater than the label of every node in its left subtree, and less than or equal to every node in its right subtree. The following are examples of, respectively, right and left strict binary search trees:
    \begin{multicols}{2}
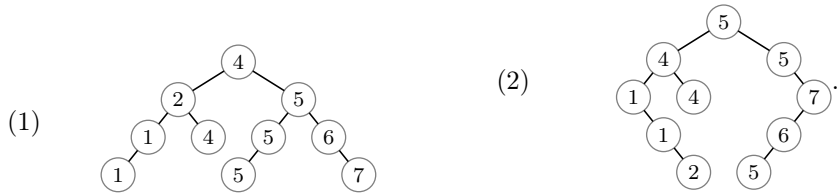
        \begin{equation}
        \label{example:rsbst}
            \begin{tikzpicture}[tinybst,baseline=-9mm]
                \node {$4$}
                child { node {$2$}
                    child { node {$1$}
                        child { node {$1$} }
                        child[missing]
                    }
                    child { node {$4$} }
                }
                child[missing]
                child { node {$5$}
                    child { node {$5$}
                        child { node {$5$} }
                        child[missing]
                    }
                    child { node {$6$}
                        child[missing]
                        child { node {$7$} }
                    }
                };
            \end{tikzpicture}
        \end{equation}
        \begin{equation}
        \label{example:lsbst}
            \begin{tikzpicture}[tinybst,baseline=-9mm]
                \node {$5$}
                child { node {$4$}
                    child { node {$1$}
                        child[missing]
                        child { node {$1$}
                            child[missing]
                            child { node {$2$} }
                        }
                    }
                    child{ node {$4$} }
                }
                child[missing]
                child { node {$5$}
                    child[missing]
                    child { node {$7$}
                        child { node {$6$}
                            child { node {$5$} }
                            child[missing]
                        }
                    child[missing]
                    }
                };
            \end{tikzpicture}.
        \end{equation}
    \end{multicols}
    
    The (left-to-right) \textit{postfix} (or \textit{postorder}) traversal of a labelled rooted binary search tree $T$ is the sequence of nodes obtained by recursively performing the postfix traversal of the left subtree of the root of $T$, then recursively performing the postfix traversal of the right subtree of the root of $T$, and then adding the root of $T$ to the sequence. The (left-to-right) \textit{postfix reading} of a labelled rooted binary search tree $T$ is the word $\mathrm{Post}(T)$ obtained by listing the labels of the nodes visited during the postfix traversal. For example, the postfix reading of the tree given in Example \ref{example:rsbst} is $1142557654$. 
    
    The (left-to-right) \textit{prefix} (or \textit{preorder}) traversal of a labelled rooted binary search tree $T$ is the sequence of nodes obtained by first adding the root of $T$ to the sequence, then recursively performing the prefix traversal of the left subtree of the root of $T$, and then recursively performing the prefix traversal of the right subtree of the root of $T$. The (left-to-right) \textit{prefix reading} of a labelled rooted binary search tree $T$ is the word $\mathrm{Pre}(T)$ obtained by listing the labels of the nodes visited during the prefix traversal. For example, the prefix reading of the tree given in Example \ref{example:lsbst} is $5411245765$.
    
    The \textit{infix} (or \textit{inorder}) traversal of a labelled rooted binary search tree $T$ is the sequence of nodes obtained by recursively performing the infix traversal of the left subtree of the root of $T$, then adding the root of $T$ to the sequence, and then recursively performing the infix traversal of the right subtree of the root of $T$. The following result is immediate from the definitions of right and left strict binary search trees:
    
    \begin{proposition}
        For any right or left strict binary search tree $T$, if a node $a$ is encountered before a node $b$ in an infix traversal, then $a \leq b$. 
    \end{proposition}
    
    In other words, the infix traversal visits nodes in weakly increasing order, in right or left strict binary search trees.
    
    Let $T$ be a right or left strict binary search tree, and let $a \in \aset$ be a symbol which labels a node of $T$. We say that a node $a$ is \textit{topmost} if all other nodes $a$ are its descendants. The following lemma, and its consequences, will be used thoroughly in this section:
    
    \begin{lemma}[{\cite[Lemma~6.6]{cm_cyclic_shifts}}]
    \label{lemma:bst_unique_path}
        Every node $a$ appears on a single path descending from the root to a leaf; thus, there is a unique topmost node $a$.
    \end{lemma}
    
    Let $T$ be a right or left strict binary search tree, and consider two nodes of $T$, respectively labelled with $a$ and $b$, for some $a,b \in \aset$. Consider the lowest common ancestor of nodes $a$ and $b$, labelled with $c$, for some $c \in \aset$. If the node $a$ is in the left subtree of the node $c$ or coincides with it, and the node $b$ is in the right subtree of the node $c$ or coincides with it, and the nodes $a$ and $b$ do not both coincide with the node $c$, then we say that the node $a$ is to \textit{the left} of the node $b$, and the node $b$ is to \textit{the right} of the node $a$. It is immediate to see that a node $a$ is to the left of a node $b$ if and only if the infix traversal visits the node $a$ before the node $b$, hence $a$ is less than or equal to $b$. Furthermore, if we consider a subset of nodes of $T$, the definition of \textit{leftmost} and \textit{rightmost} nodes follows naturally.

    
        
    The following algorithm allows us to insert a letter from $\aset$ into an existing right strict binary search tree, as a leaf node in the unique place that maintains the property of being a right strict binary search tree:
        
    \begin{algorithm}[h]
        \DontPrintSemicolon
        \KwIn{A right strict binary search tree $T$ and a symbol $a \in \aset$.}     
        \KwOut{A right strict binary search tree $T \leftarrow a$.}
        \BlankLine
        \eIf{$T$ is empty,}{create a node and label it $a$;}{examine the label $x$ of the root node; if $a > x$, recursively insert $a$ into the right subtree of the root node; otherwise recursively insert $a$ into the left subtree of the root node;}
        \Return the resulting tree.
        \caption{\textit{Right strict leaf insertion}.}
        \label{alg_rsbst}
    \end{algorithm}

    Let $u \in \aset^*$. Using the insertion algorithm above, we can compute a unique right strict binary search tree $\psymb_{\sylv} (u)$ from $u$: we start with the empty tree and insert the letters of $u$, one-by-one from right-to-left. Notice that, for any right strict binary search tree $T$, we have that $\psymb_{\sylv}\left( \mathrm{Post}(T) \right) = T$, that is, the right strict insertion algorithm, with the postfix reading of $T$ as input, gives back $T$. As such, any right strict binary search tree can be seen as an output of the right strict insertion algorithm.
    
    We define the relation $\equiv_{\sylv}$ on $\aset^*$ as follows: For $u,v \in \aset^*$,
    \begin{center}
        $u \equiv_{\sylv} v \iff \psymb_{\sylv}(u) = \psymb_{\sylv}(v)$.
    \end{center}
    This relation is a congruence on $\aset^*$, called the \textit{sylvester congruence}. The factor monoid $\aset^*/{\equiv_{\sylv}}$ is the infinite-rank \textit{sylvester monoid}, denoted by $\sylv$. The congruence $\equiv_{\sylv}$ naturally restricts to a congruence on $\anset^*$, and the factor monoid $\anset^*/{\equiv_{\sylv}}$ is the \textit{sylvester monoid of rank} $n$, denoted by $\sylvn$.
    
    It follows from the definition of $\equiv_{\sylv}$ that each element $[u]_{\sylv}$ of $\sylv$ can be identified with the combinatorial object $\psymb_{\sylv} (u)$. As such, for each right strict binary search tree $T$, the set of words $u \in \aset^*$ such that $\psymb_{\sylv}(u) = T$ is called the \textit{sylvester class} of $T$, and the postfix reading of $T$ is called the \textit{canonical word} of the sylvester class of $T$.
    
    Recall that the content of $u$ describes the number of occurrences of each letter of $\aset$ in $u$. It is immediate from the definition of the sylvester monoid that if $u \equiv_{\sylv} v$, then $c(u) = c(v)$. Thus, we can define the content of an element of $\sylv$ as the content of any word which represents it. Furthermore, since $c(u) = c(v)$ implies that $supp(u) = supp(v)$, we can also define the support of an element of $\sylv$ as the support of any word which represents it. We define the content and support of a right strict binary search tree as the content and support of its corresponding sylvester class.
    
    Notice that $\sylv_n$ is a submonoid of $\sylv$, for each $n \in \nset$, and, for $n,m \in \nset$, if $n \leq m$, then $\sylv_n$ is a submonoid of $\sylv_m$.
	
    The sylvester monoid can also be defined by the presentation $\left\langle \aset \mid \mathcal{R}_{\sylv} \right\rangle$, where
    \begin{align*}
        \mathcal{R}_{\sylv} =& \left\{ (caub,acub): a \leq b < c, u \in \aset^* \right\}.
    \end{align*}
    These defining relations are known as the \textit{sylvester relations}. A presentation for the sylvester monoid of rank $n$, for some $n \in \nset$, can be obtained by restricting generators and relations of the above presentation to generators in $\anset$.
    
	We now give an alternative characterization of the sylvester monoid, inspired by the characterization of the hypoplactic monoid using inversions (see \cite{novelli_hypoplactic,cmr_hypo_id}). Let $u \in \aset^*$ and let $a,b \in supp(u)$ be such that $a < b$. We say $u$ has a $b$-$a$ \textit{right precedence} if, when reading $u$ from right to left, $b$ occurs before the first occurrence of $a$ and, for any $c \in supp(u)$ such that $a < c < b$, $c$ does not occur before the first occurrence of $a$. The number of occurrences of $b$ before the first occurrence of $a$ is the \textit{index} of the right precedence.
    
    Notice that, by the definition of a right precedence, for any given $a \in supp(u)$, there is at most one $b \in supp(u)$ such that $u$ has a $b$-$a$ right precedence (of index $k$, for some $k \in \nset$). On the other hand, $u$ can have several right precedences of the form $b$-$x$, for a fixed $b$. 
    
    \begin{example}
        The word $3123$ has a $2$-$1$ and a $3$-$2$ right precedence, both of index $1$, however, it does not have a $3$-$1$ right precedence, since $2$ occurs before the first occurrence of $1$; the word $2313$ has a $3$-$1$ right precedence of index $1$ and a $3$-$2$ right precedence of index $2$; and the word $3132$ has a $2$-$1$ right precedence of index $1$, and does not have a $3$-$1$ right precedence.
    \end{example}
    
    In order to prove that the sylvester monoid can be characterized using only the content and right precedences of words, we require some lemmata:
    
    
    
    \begin{lemma}
    \label{lemma:right_precedence_ancestor_nodes}
        Let $u \in \aset^*$, and let $a,b \in supp \left( u \right)$ be such that $a<b$ and $b$ occurs at least $k$ times in $u$, for some $k \in \nset$. Then, $u$ has a $b$-$a$ right precedence of index $k$ if and only if the topmost node $a$ in $\psymb_{\sylv} (u)$ has exactly $k$ ancestor nodes labelled with $b$, and no ancestor nodes labelled with $c$, for any $c \in supp \left( \psymb_{\sylv} (u) \right)$ such that $a<c<b$.
    \end{lemma}
    
    \begin{proof}
        Suppose that $u$ has a $b$-$a$ right precedence of index $k$. It is clear that, as a consequence of the insertion algorithm~\ref{alg_rsbst}, the topmost node $a$ in $\psymb_{\sylv} (u)$ corresponds to the rightmost occurrence of $a$ in $u$. Thus, there are exactly $k$ symbols $b$ inserted before $a$, when computing $\psymb_{\sylv} (u)$. Notice that, by Lemma~\ref{lemma:bst_unique_path}, the corresponding nodes must be in a single path from the root to any leaf of $\psymb_{\sylv} (u)$. Furthermore, since no symbol $c$ occurs before the rightmost symbol $a$, when reading $u$ from right-to-left, for any $a < c < b$, we have that the rightmost symbol $a$ must be inserted as a left child of a node $b$, in particular, the node corresponding to the $k$-th inserted symbol $b$. This is due to the fact that, during the ``searching'' step of the insertion algorithm, the symbol $a$ will satisfy exactly the same criteria as the last inserted symbol $b$, except when checking the $k$-th node $b$. Thus, the topmost node $a$ in $\psymb_{\sylv} (u)$ must have exactly $k$ ancestor nodes labelled with $b$, and no ancestor nodes labelled with $c$.
        
        On the other hand, suppose $u'$ is such that $u \equiv_{\sylv} u'$. Suppose, in order to obtain a contradiction, that $u'$ does not have a $b$-$a$ right precedence of index $k$. Then, by the previous part of this proof, we have that $\psymb_{\sylv} (u') \neq \psymb_{\sylv} (u)$, which contradicts our hypothesis.
    \end{proof}
    
    It is clear, from the previous lemma, that all words in the same sylvester class must share exactly the same right precedences. However, it is not immediate that two words which share the same content and right precedences will produce exactly the same output, when considered as inputs for the insertion algorithm.
    
    \begin{lemma}
    \label{lemma:right_precedence_left_child}
        Let $T$ be a right strict binary search tree, and let $a,b \in supp(T)$, with $a$ strictly less than $b$, be such that the topmost node $a$ is a left child of a node $b$. Then, all words in the sylvester class of $T$ have a $b$-$a$ right precedence of index $k$, where $k$ is the number of ancestor nodes of the topmost node $a$ labelled with $b$.
    \end{lemma}
    
    \begin{proof}
        If the topmost node $a$ had an ancestor node labelled with $c$, for some $a < c < b$, then the node $a$ would be in the left subtree of a node $c$, but even more so, since the node $a$ is a left child of the node $b$, then the node $b$ would also be in the left subtree of a node $c$, which contradicts the hypothesis that $c$ is strictly less than $b$. The result follows from Lemma~\ref{lemma:right_precedence_ancestor_nodes}.
    \end{proof}
    
    More generally, we can also see that if a topmost node $a$ is in the left subtree of some node $b$, then the words in the sylvester class of $T$ will have a $c$-$a$ right precedence, for some $c \in supp(T)$ such that $a < c \leq b$.
    
    Suppose $u,v \in \aset^*$ are words which share the same content, but $\psymb_{\sylv} (u) \neq \psymb_{\sylv} (v)$. It is immediate that at least more than one different symbol must occur in $u$ and $v$. Furthermore, if the roots are labelled differently, it is clear that the words in the sylvester class of the tree whose root label is higher will have a right precedence concerning the root of the other tree, while the words in the sylvester class of the tree whose root label is lower will not. Thus, we will consider that, up to a certain depth, the two trees $\psymb_{\sylv} (u)$ and $\psymb_{\sylv} (v)$ are identical.
    
    Let $T_{u,v}$ be the right strict binary search tree obtained by removing all nodes of decreasing depth in $\psymb_{\sylv} (u)$ and $\psymb_{\sylv} (v)$ until we obtain the same tree. 
    
    \begin{lemma}
    \label{lemma:existence_of_child_node}
        Any node of $\psymb_{\sylv} (u)$ and $\psymb_{\sylv} (v)$, corresponding to a leaf of highest depth in $T_{u,v}$, has a left (respectively, right) child in $\psymb_{\sylv} (u)$ if and only if it has a left (respectively, right) child in $\psymb_{\sylv} (v)$.
    \end{lemma}
    
    \begin{proof}
        Suppose, in order to obtain a contradiction, that there is a node in $\psymb_{\sylv} (u)$ and $\psymb_{\sylv} (v)$, corresponding to a leaf of highest depth in $T_{u,v}$ and labelled with $a$, for some $a \in supp\left(T_{u,v}\right)$, which has a child node in $\psymb_{\sylv} (u)$ that does not occur in $\psymb_{\sylv} (v)$. Assume, without loss of generality, that it is a left child. Notice that the node $a$ cannot be the only leaf of highest depth in $T_{u,v}$. Furthermore, this node cannot correspond to the leftmost leaf of highest depth in $T_{u,v}$, otherwise, this would imply that $\psymb_{\sylv} (u)$ has at least one node labelled with a symbol which either cannot occur in $\psymb_{\sylv} (v)$, or occurs more times in $\psymb_{\sylv} (u)$ than in $\psymb_{\sylv} (v)$, which contradicts our hypothesis that $u$ and $v$ share the same content.
        
        Let $b$ be the label of the left child of the node $a$, in $\psymb_{\sylv} (u)$; let $c$ be the label of the node corresponding to the rightmost leaf of highest depth in $T_{u,v}$ to the left of the node $a$; let $d$ be the label of the lowest common ancestor of the nodes $a$ and $c$. Notice that the nodes $a$ and $b$ are in the right subtree of the node $d$, hence $c \leq d < b \leq a$. On the other hand, the infix traversal, in $\psymb_{\sylv} (u)$, first visits the node $c$, then the node $d$, then the node $b$, and immediately afterwards the node $a$. Since $u$ and $v$ have the same content, the infix traversal, in $\psymb_{\sylv} (v)$, must also first visit the node $c$, then all nodes labelled with $b$, and then the node $a$. As such, since in $\psymb_{\sylv} (v)$, the node $a$ does not have a left child, and the node $c$ corresponds to the rightmost leaf of highest depth in $T_{u,v}$ to the left of the node $a$, all nodes labelled with $b$ must be descendants of the node $c$. But this implies that, in $\psymb_{\sylv} (v)$, all nodes labelled with $b$ are in the left subtree of the node $d$, hence $b \leq d$, which contradicts our hypothesis.
    \end{proof}
    
    This lemma shows that two different right strict binary search trees are identical up to a certain depth, and the first difference between the trees, when searching in depth, is a difference between the labels of some nodes. Now, we are ready to prove the right precedence characterization of the sylvester monoid:   

    \begin{proposition}
    \label{proposition:right_precedence_characterization}
        For $u,v \in \anset^*$, we have that $u \equiv_{\sylv} v$ if and only if $u$ and $v$ share exactly the same content and right precedences.
    \end{proposition}
    
    \begin{proof}
        We already know that if two words $u$ and $v$ are in the same sylvester class, then they share the same content and right precedences, by Lemma~\ref{lemma:right_precedence_ancestor_nodes}. Suppose now that $u$ and $v$ share the same content, but $\psymb_{\sylv} (u) \neq \psymb_{\sylv} (v)$. Once again, consider the right strict binary search tree $T_{u,v}$, obtained by removing all nodes of decreasing depth in $\psymb_{\sylv} (u)$ and $\psymb_{\sylv} (v)$ until we obtain the same tree. By the previous lemma and our hypothesis, there must exist a node in $\psymb_{\sylv} (u)$ and $\psymb_{\sylv} (v)$, corresponding to a leaf of highest depth in $T_{u,v}$, which has a child node labelled differently in $\psymb_{\sylv} (u)$ than in $\psymb_{\sylv} (v)$. 
        
        Let $a$ be the label of the child node in $\psymb_{\sylv} (u)$ and $b$ be the label in $\psymb_{\sylv} (v)$. Assume, without loss of generality, that $a < b$. Notice that the node $a$ in $\psymb_{\sylv} (u)$ must be topmost, otherwise, the node $b$ in $\psymb_{\sylv} (v)$ would be in the left subtree of a node labelled with $a$. Let $c$ be the label of the parent node of node $a$ in $\psymb_{\sylv} (u)$ and node $b$ in $\psymb_{\sylv} (v)$. Using an argument similar to the one used in the proof of Lemma~\ref{lemma:existence_of_child_node}, we can see that the topmost node $a$ is also a descendant of the node $c$, in $\psymb_{\sylv} (v)$. Even more so, the topmost node $a$ must be in the left subtree of the node $b$, in $\psymb_{\sylv} (v)$.
        
        Suppose nodes $a$ and $b$ are left children. First of all, notice that $a < b \leq c$. Furthermore, by Lemma~\ref{lemma:right_precedence_left_child}, we have that $u$ has a $c$-$a$ right precedence of index $k$, where $k$ is the number of ancestor nodes of the node $a$ labelled with $c$. If $b$ is equal to $c$, then $v$ either has a $c$-$a$ right precedence of index strictly greater than $k$, or it does not have a $c$-$a$ right precedence at all; otherwise, if $b$ is strictly less than $c$, then $v$ does not have a $c$-$a$ right precedence.
        
        Suppose nodes $a$ and $b$ are right children. If $u$ has a $d$-$a$ right precedence of index $k$, for some $d \in supp\left(T_{u,v}\right)$ and some $k \in \nset$, then the node $a$ is in the left subtree of a node $d$, in $\psymb_{\sylv} (u)$ as well as in $\psymb_{\sylv} (v)$, and the node $b$ is in the left subtree of that node $d$ in $\psymb_{\sylv} (v)$. If $b$ is equal to $d$, then $v$ either has a $d$-$a$ right precedence of index greater than $k$, or it does not have a $d$-$a$ right precedence at all; otherwise, if $b$ is strictly less than $d$, then $v$ does not have a $d$-$a$ right precedence. On the other hand, if $u$ does not have a $d$-$a$ right precedence, for any $d \in supp\left(T_{u,v}\right)$, then $v$ has a $d'$-$a$ right precedence, for some symbol $d'$ which is either equal to $b$ or labels an ancestor node of the topmost node $a$ in the left subtree of the node $b$, in $\psymb_{\sylv} (v)$.
        
        Thus, we can conclude that if $u$ and $v$ share the same content, but $\psymb_{\sylv} (u) \neq \psymb_{\sylv} (v)$, then $u$ and $v$ do not share the same right precedences.
    \end{proof}
    
	In light of the previous result, we say that an element $[u]_{\sylv}$ of $\sylv$ has a $b$-$a$ right precedence (of index $k$) if the word $u$ itself, and hence any other word in $[u]_{\sylv}$, has a $b$-$a$ right precedence (of index $k$).

    The sylvester and hypoplactic monoids are closely related (see \cite{priez_binary_trees,cm_sylv_crystal}). In fact, the hypoplactic monoid is a homomorphic image of the sylvester monoid. Notice that, for any right strict binary search tree with support $\{a_1 < \dots < a_k\}$, all words in its sylvester class have a $a_{i+1}$-$a_i$ inversion if and only if a node $a_{i+1}$ appears in a right subtree of a node $a_i$. Hence, the map defined by
    \[
    [u]_{\sylv} \mapsto [u]_{\hypo},
    \]
    for $u \in \aset$, is a natural homomorphism of $\sylv$ into $\hypo$, though not an embedding. Therefore, $\hypo$ is in the variety generated by $\sylv$, and thus must satisfy all identities satisfied by $\sylv$.

    The insertion algorithm for left strict binary search trees is dual to Algorithm \ref{alg_rsbst}:

    \begin{algorithm}[h]
        \DontPrintSemicolon
        \KwIn{A left strict binary search tree $T$ and a symbol $a \in \aset$.}     
        \KwOut{A left strict binary search tree $T \leftarrow a$.}
        \BlankLine
        \eIf{$T$ is empty,}{create a node and label it $a$;}{examine the label $x$ of the root node; if $a < x$, recursively insert $a$ into the left subtree of the root node; otherwise recursively insert $a$ into the right subtree of the root node;}
        \Return the resulting tree.
        \caption{}
        \label{alg_lsbst}
    \end{algorithm}
    
    Let $u \in \aset^*$. Using the insertion algorithm above, we can compute a unique left strict binary search tree $\psymb_{\sylvh} (u)$ from $u$: we start with the empty tree and insert the symbols of $u$, one-by-one from left-to-right. Notice that, for any left strict binary search tree $T$, we have that $\psymb_{\sylvh}\left( \mathrm{Pre}(T) \right) = T$, that is, the left strict insertion algorithm, with the prefix reading of $T$ as input, gives back $T$. As such, any left strict binary search tree can be seen as an output of the left strict insertion algorithm.
    
    We define the \textit{\#-sylvester congruence} $\equiv_{\sylvh}$, the infinite-rank \textit{\#-sylvester monoid} $\sylvh$, the \textit{\#-sylvester monoid of rank} $n$ $\sylvhn$, and the \textit{content} and \textit{support} of a \#-sylvester class in a similar fashion as before.
    
    The \#-sylvester monoid can also be defined by the presentation $\left\langle \aset \mid \mathcal{R}_{\sylvh} \right\rangle$, where
    \begin{align*}
        \mathcal{R}_{\sylvh} =& \left\{ (buac,buca): a < b \leq c, u \in \aset^* \right\}.
    \end{align*}
    These defining relations are known as the \textit{\#-sylvester relations}.
    
    The sylvester and \#-sylvester monoids of finite rank $n$ are anti\chyph isomorphic: a natural anti-isomorphism of $\sylvn$ into $\sylvhn$ arises by taking a right strict binary search tree, reflecting it about a vertical axis, and renumbering the label $i$ of each node to $n-i+1$, thus obtaining a left strict binary search tree. Similarly, we can define a natural anti-isomorphism of $\sylvhn$ into $\sylvn$. These anti-isomorphisms allow us to easily deduce results for the \#-sylvester monoid from results for the sylvester monoids. However, notice that these natural anti-isomorphisms do not arise in the infinite rank case, as there is no natural way to renumber the labels of the nodes. In fact, we have the following:
    
    \begin{proposition}
    \label{proposition:no_anti-isomorphism_sylv_sylvh}
        There is no anti-isomorphism of $\sylv$ into $\sylvh$.
    \end{proposition}
    
    \begin{proof}
        Suppose, in order to obtain a contradiction, that there exists such an anti-isomorphism $\sigma: \sylv \rightarrow \sylvh$. Notice that, for any element of the form $[a]_{\sylv}$, where $a \in \aset$, we have that $\left| \sigma \left( [a]_{\sylv} \right) \right| = 1$, since $\left| \sigma \left( [a]_{\sylv} \right) \right| > 1$ implies that 
        \[
            \left| [a]_{\sylv} \right| = \left| \sigma^{-1} \left( \sigma \left( [a]_{\sylv} \right) \right) \right| > 1,
        \]
        along with the fact that the identity of $\sylv$ is mapped to the identity of $\sylvh$. Thus, $\sigma$ maps the $\sylv$-classes of generators into $\sylvh$-classes of generators.
        
        Thus, we have that $\sigma \left( [1]_{\sylv} \right) = [x]_{\sylvh}$, for some $x \in \aset$. Since we are considering the infinite-rank case, there must exist $b,y \in \aset$ such that $y > x$ and $\sigma \left( [b]_{\sylv} \right) = [y]_{\sylvh}$. Notice that, since $\sigma$ maps $1$ to $x$, then $b > 1$. As such, we have that $b11 \equiv_{\sylv} 1b1$, but on the other hand,
        \[
            \sigma \left( [b11]_{\sylv} \right) = [xxy]_{\sylvh} \neq [xyx]_{\sylvh} = \sigma \left( [1b1]_{\sylv} \right),
        \]
        which contradicts the hypothesis that $\sigma$ is an anti-isomorphism.
    \end{proof}
    
	We now give an alternative characterization of the \#-sylvester monoid, parallel to the one given for the sylvester monoid. Let $u \in \aset^*$ and let $a,b \in supp(u)$ be such that $a < b$. We say $u$ has a $a$-$b$ \textit{left precedence} of index $k$ if, when reading $u$ from left to right, $a$ occurs before the first occurrence of $b$ and, for any $c \in supp(u)$ such that $a < c < b$, $c$ does not occur before the first occurrence of $b$. The number of occurrences of $a$ before the first occurrence of $b$ is the \textit{index} of the left precedence.
    
    Notice that, by the definition of a left precedence, for any given $b \in supp(u)$, there is at most one $a \in supp(u)$ such that $u$ has a $a$-$b$ left precedence (of index $k$, for some $k \in \nset$). On the other hand, $u$ can have several left precedences of the form $a$-$x$, for a fixed $a$.
    
    \begin{example}
        The word $1231$ has a $1$-$2$ and a $2$-$3$ left precedence, both of index $1$, while $1312$ has a $1$-$2$ left precedence of index $2$ and a $1$-$3$ left precedence of index $1$, and $3121$ has a $1$-$2$ left precedence of index $1$.
    \end{example}
    
    The following proposition is parallel to Proposition~\ref{proposition:right_precedence_characterization}:
    
    \begin{proposition}
    \label{proposition:left_precedence_characterization}
        For $u,v \in \anset^*$, we have that $u \equiv_{\sylvh} v$ if and only if $u$ and $v$ share exactly the same content and left precedences.
    \end{proposition}

	In light of the previous result, we say that an element $[u]_{\sylvh}$ of $\sylvh$ has a $a$-$b$ left precedence (of index $k$) if the word $u$ itself, and hence any other word in $[u]_{\sylvh}$, has a $a$-$b$ left precedence (of index $k$).

    The hypoplactic monoid is also a homomorphic image of the \#-sylvester monoid. Hence, $\hypo$ is in the variety generated by $\sylvh$, and thus must satisfy all identities satisfied by $\sylvh$.

	
	\subsection{The Baxter monoid}
    \label{subsection:baxt}
    
    This subsection gives a brief overview of the Baxter monoid and its related combinatorial object and insertion algorithm, as well as results from \cite{cm_identities}. Due to its connection with the sylvester and \#-sylvester monoids, we also give an alternative characterization of this monoid, derived from the characterization of $\sylv$ and $\sylvh$ given in the previous subsection. For more information, see \cite{giraudo_baxter2} and \cite{cm_sylv_crystal}.
        
    The \textit{canopy} of a rooted binary tree $T$ is the word over $\{0,1\}$ obtained by doing an infix traversal of $T$, outputting $1$ when an empty left subtree is encountered and $0$ when an empty right subtree is encountered, then omitting the first and last symbols of the resulting word (which correspond, respectively, to the empty left subtree of the leftmost node and the empty right subtree of the rightmost node). 
    
    A pair of twin binary search trees consist of a left strict binary search tree $T_L$ and a right strict binary search tree $T_R$ , such that $T_L$ and $T_R$ contain the same symbols, and the canopies of $T_L$ and $T_R$ are complementary, in the sense that the $i$-th symbol of the canopy of $T_L$ is $0$ (respectively $1$) if and only if the $i$-th symbol of the canopy of $T_R$ is $1$ (respectively $0$). The following is an example of a pair of twin binary search trees:
    
    \begin{center}
        $\left(
        \begin{tikzpicture}[tinybst,baseline=-9mm]
            \node {$5$}
            child { node {$4$}
                child { node {$1$}
                    child[missing]
                    child { node {$1$}
                        child[missing]
                        child { node {$2$} }
                    }
                }
                child{ node {$4$} }
            }
            child[missing]
            child { node {$5$}
                child[missing]
                child { node {$7$}
                    child { node {$6$}
                        child { node {$5$} }
                        child[missing]
                    }
                child[missing]
                }
            };
        \end{tikzpicture}
        ,
        \begin{tikzpicture}[tinybst,baseline=-9mm]
            \node {$4$}
            child { node {$2$}
                child { node {$1$}
                    child { node {$1$} }
                    child[missing]
                }
                child { node {$4$} }
            }
            child[missing]
            child { node {$5$}
                child { node {$5$}
                    child { node {$5$} }
                    child[missing]
                }
                child { node {$6$}
                    child[missing]
                    child { node {$7$} }
                }
            };
        \end{tikzpicture}
        \right)$      
    \end{center}
    
    Let $u \in \aset^*$. Due to {\cite[Proposition~4.5]{giraudo_baxter2}}, the pair of binary seach trees $( \psymb_{\sylvh} (u), \psymb_{\sylv} (u) )$ is a pair of \textbf{twin} binary search trees. As such, by defining $\psymb_{\baxt} (u)$ as this pair, we can use Algorithms \ref{alg_rsbst} and \ref{alg_lsbst} to compute a unique pair of twin binary search trees $\psymb_{\baxt} (u)$ from $u$.
        
    We define the \textit{Baxter congruence} $\equiv_{\baxt}$, the infinite-rank \textit{Baxter monoid} $\baxt$, the \textit{Baxter monoid of rank} $n$ $\baxtn$, and the \textit{content} and \textit{support} of a Baxter class in a similar fashion as before.
    
    The Baxter monoid can also be defined by the presentation $\left\langle \aset \mid \mathcal{R}_{\baxt} \right\rangle$, where
    \begin{align*}
        \mathcal{R}_{\baxt} =& \left\{ (cudavb,cuadvb): a \leq b < c \leq d, u,v \in \aset^* \right\} \\
        & \cup \left\{ (budavc,buadvc): a < b \leq c < d, u,v \in \aset^* \right\}.
    \end{align*}
    These defining relations are known as the \textit{Baxter relations}.
    
    The Baxter, sylvester and \#-sylvester monoids are closely related, due to {\cite[Proposition~3.7]{giraudo_baxter2}}: For $u,v \in \aset^*$, we have that $u \equiv_{\baxt} v$ if and only if $u \equiv_{\sylv} v$ and $u \equiv_{\sylvh} v$. As a consequence of this, and Propositions~\ref{proposition:right_precedence_characterization} and \ref{proposition:left_precedence_characterization}, we have the following:
    
    \begin{corollary}
    \label{corollary:baxter_precedence_characterization}
        For $u,v \in \anset^*$,  $u \equiv_{\baxt} v$ if and only if $u$ and $v$ share exactly the same content and left and right precedences.
    \end{corollary}

	In light of the previous result, we say that an element $[u]_{\baxt}$ of $\baxt$ has a $b$-$a$ right (respectively, left) precedence of index $k$ if the word $u$ itself, and hence any other word in $[u]_{\baxt}$, has a $b$-$a$ right (respectively, left) precedence of index $k$.

    It is also easy to see that the maps defined by
    \[
    [u]_{\baxt} \mapsto [u]_{\sylv} \quad \text{and} \quad [u]_{\baxt} \mapsto [u]_{\sylvh},
    \]
    for $u \in \aset$, are natural homomorphisms of $\baxt$ into $\sylv$ and $\sylvh$, respectively (though not embeddings). Therefore, both $\sylv$ and $\sylvh$ are monoids in the variety generated by $\baxt$, and thus must satisfy all identities satisfied by $\baxt$. 
    
	
	\section{Embeddings}
	\label{section:embeddings}
	
	
	In this section, we prove that the sylvester monoids of ranks greater than or equal to $2$ satisfy exactly the same identities. We do this by constructing embeddings of sylvester monoids of any rank greater than $2$ into direct products of copies of the sylvester monoid of rank $2$, as it is not possible to embed one into another. Thus, they generate the same variety and, by Birkhoff's Theorem, satisfy exactly the same identities. We also show that the basis rank of the variety generated by $\sylv$ is $2$.

	By parallel reasoning, we prove the same results for the \#-sylvester monoids of ranks greater than or equal to $2$. Furthermore, as a consequence of {\cite[Proposition~3.7]{giraudo_baxter2}}, we also obtain the same results for the Baxter monoids of ranks greater than or equal to $2$.	
	
	\subsection{Non-existence of embedding into a monoid of lesser rank}
    \label{subsection:no_embedding_n_nm}
    
    It is not possible to embed a sylvester monoid of finite rank into a sylvester monoid of lesser rank:
	
	\begin{proposition}
    \label{proposition:no_emb_sylvn_sylvm}
		For all $n > m \geq 1$, there is no embedding of $\sylvn$ into $\sylv_m$.
	\end{proposition}
	
	\begin{proof}
		First of all, notice that $\sylv_1$ is isomorphic to the free monogenic monoid and $\sylvn$ is non\chyph commutative, for any $n \geq 2$. Thus, there is no embedding of $\sylvn$ into $\sylv_1$.
		
		On the other hand, if there exists an embedding of $\sylvn$ into $\sylv_m$, for some $n > m \geq 2$, then, since $\sylv_m$ is a submonoid of $\sylvnm$, there must also exist an embedding of $\sylvn$ into $\sylvnm$. As such, we just need to prove that this second embedding cannot exist.
		
		Suppose, in order to obtain a contradiction, that there exists $n \geq 3$ such that we have an embedding $\phi: \sylvn \longmapsto \sylvnm$. Without loss of generality, suppose $n$ is the smallest positive integer in such conditions.
		
		Observe that $supp \left( \phi \left( \left[2 \cdots n \right]_{\sylvn} \right) \right) = \anmset$, that is, the image of the product of all generators of $\sylvn$, except for $1$, has all the possible letters of $\anmset$. Indeed, if $supp \left( \phi \left( \left[2 \cdots n \right]_{\sylvn} \right) \right) \subsetneqq \anmset$, we would be able to construct an embedding from the submonoid isomorphic to $\sylvnm$ of $\sylvn$, generated by all generators of $\sylvn$ except for $1$, into a submonoid of $\sylvnm$ isomorphic to $\sylv_{n-2}$. This contradicts the minimality of $n$.
		
		Hence, since all letters of $\anmset$ already occur in $\phi \left( \left[2 \cdots n \right]_{\sylvn} \right)$, if we multiply this element by any other element of $\sylvnm$ to the left, we obtain an element with the same right precedences as $\phi \left( \left[2 \cdots n \right]_{\sylvn} \right)$. Thus, by Proposition~\ref{proposition:right_precedence_characterization}, we have that 
		\[
		\phi \left( [12]_{\sylvn} \cdot \left[2 \cdots n \right]_{\sylvn} \right) = \phi \left( [21]_{\sylvn} \cdot \left[2 \cdots n \right]_{\sylvn} \right).
		\]
		
		On the other hand, we have that
		\[
		[12]_{\sylvn} \cdot \left[2 \cdots n \right]_{\sylvn} \neq [21]_{\sylvn} \cdot \left[2 \cdots n \right]_{\sylvn},
		\]
		since the left-hand side has a $2$-$1$ right precedence of index $2$, and the right-hand side has a $2$-$1$ right precedence of index $1$.
        
        This contradicts our hypothesis that $\phi$ is injective. Hence, for all $n \geq 2$, there is no embedding of $\sylvn$ into $\sylvnm$. As such, there is no embedding of $\sylvn$ into $\sylv_m$, for $n > m \geq 2$.
	\end{proof}
	
	\begin{corollary}
    \label{corollary:no_emb_sylv_into_sylvn}
		There is no embedding of $\sylv$ into $\sylvn$, for any $n \in \nset$. 
	\end{corollary}
	
	\begin{proof}
		If such an embedding existed, for some $n \in \nset$, then, by restricting the embedding to the first $n+1$ generators of $\sylv$, we would obtain an embedding of $\sylvnp$ into $\sylvn$, which contradicts the previous proposition.
	\end{proof}
	
	If there existed an embedding of a \#-sylvester monoid of finite rank into a \#-sylvester monoid of lesser rank, then we would be able to compose it with the natural anti-isomorphisms, thus obtaining an embedding for the sylvester case. Therefore, no such embedding exists, as well as no embedding of the infinite-rank \#-sylvester monoid into a \#-sylvester monoid of finite rank. We can also prove the corresponding result for the Baxter monoid:
	
	\begin{proposition}
    \label{proposition:no_emb_baxtn_baxtnm}
		For all $n > m \geq 1$, there is no embedding of $\baxtn$ into $\baxt_m$.
	\end{proposition}
	
	\begin{proof}
		The proof follows the same reasoning as given in the proof of Proposition~\ref{proposition:no_emb_sylvn_sylvm}: Instead of considering the elements 
		\[
		[12]_{\sylvn} \cdot \left[2 \cdots n \right]_{\sylvn} \quad \text{and} \quad [21]_{\sylvn} \cdot \left[2 \cdots n \right]_{\sylvn}, 
		\]
		we consider, respectively, the elements 
		\[
		\left[2 \cdots n \right]_{\baxtn} \cdot \left[12 \right]_{\baxtn} \cdot \left[2 \cdots n \right]_{\baxtn} \quad \text{and} \quad \left[2 \cdots n \right]_{\baxtn} \cdot \left[21 \right]_{\baxtn} \cdot \left[2 \cdots n \right]_{\baxtn}.
		\]
	\end{proof}
	
	\begin{corollary}
    \label{corollary:no_emb_baxt_into_baxtn}
		There is no embedding of $\baxt$ into $\baxtn$, for any $n \in \nset$. 
	\end{corollary}

	
	\subsection{Embedding into a direct product of copies of monoids of rank 2}
    \label{subsection:embedding_direct_product_rank2}
	
	Although it is not possible to embed the sylvester, \#-sylvester and Baxter monoids, of rank greater than or equal to $2$, into their corresponding monoids of rank $2$, we now show how to embed each of them into a direct product of copies of their corresponding monoid of rank $2$, starting with the sylvester monoid. For simplicity, we shall denote by $M^k$ the direct product of $k$ copies of a monoid $M$.
	
	For any $i,j \in \aset$, with $i < j$, define a map from $\aset$ to $\sylv_2$ in the following way: For any $a \in \aset$,
	\begin{align*} 
        a &\longmapsto \begin{cases}
            [1]_{\sylv_2} & \text{if } a = i;\\
            [2]_{\sylv_2} & \text{if } a = j;\\
            [21]_{\sylv_2} & \text{if } i < a < j;\\
            \left[\varepsilon\right]_{\sylv_2} & \text{otherwise};
        \end{cases}
	\end{align*}
	and extend it to a homomorphism $\varphi_{ij}: \aset^* \longrightarrow \sylv_2$, in the usual way. This homomorphism is analogous to the one given in \cite[Subsection~3.2]{cmr_hypo_id}. The proofs of the following lemmata and propositions make use of the new characterization using right precedences for the sylvester monoid.
	
	Notice that $\varphi_{ij} (w)$ is the sylvester class of the word obtained from $w$ by replacing any occurrence of $i$ by $1$; any occurrence of $j$ by $2$; any occurrence of an $a$, with $i < a < j$, by $21$; and erasing any occurrence of any other element.
	
	\begin{lemma}
	\label{lemma:phi_factor_homomorphism}
		$\varphi_{ij}$ factors to give a homomorphism $\varphi_{ij} : \sylv \longrightarrow \sylv_2$.
	\end{lemma}
	
	\begin{proof}
	    Since $\sylv$ is given by the presentation $\left\langle \aset \mid \mathcal{R}_{\sylv} \right\rangle$, we just need to verify that both sides of the sylvester relations have the same image under $\varphi_{ij}$. 
	    
	    Let $a,b,c \in \aset$ and $u \in \aset^*$ be such that $a \leq b < c$. If $\varphi_{ij}$ maps either $a$ or $c$ to $\left[\varepsilon\right]_{\sylv_2}$, then the images of $caub$ and $acub$ under $\varphi_{ij}$ coincide. Assume, without loss of generality, that $\varphi_{ij}$ does not map any letter to $\left[\varepsilon\right]_{\sylv_2}$. Then, $\varphi_{ij}$ maps $a$ to $[1]_{\sylv_2}$, $b$ to either $[21]_{\sylv_2}$ or $[1]_{\sylv_2}$, and $c$ to $[2]_{\sylv_2}$. Notice that $b$ is mapped to an element with no right precedences. As such, we have that
        \[
            \varphi_{ij} (caub) = \varphi_{ij} (acub),
        \]
        since both sides of the equality have no right precedences. Hence, $\mathcal{R}_{\sylv} \subseteq \ker \varphi_{ij}$.
	\end{proof}
	
	Let $w \in \anset^*$, for some $n \geq 3$. Suppose $supp \left( w \right) = \{a_1 < \dots < a_m\}$, for some $m \in \nset$. Observe that, ranging $1 \leq i < m$, we can get from the maps $\varphi_{a_i a_{i+1}}$ the number of occurrences of $a_i$ and $a_{i+1}$ in $w$, since, when reading any word in $\varphi_{a_i a_{i+1}}([w]_{\sylv})$, every occurrence of $1$ corresponds exactly to an occurrence of $a_i$ in $w$, and every occurrence of $2$ corresponds exactly to an occurrence of $a_{i+1}$. 
	
	Recall that there is at most one index $j$, with $i < j \leq m$, such that $w$ has a $a_j$-$a_i$ right precedence. Ranging $1 \leq i < j \leq m$, we can also check if $w$ has a $a_j$-$a_i$ right precedence: If $w$ does not have any $b$-$a_i$ right precedence, for $b < a_j$, then no $b$ occurs before $a_i$, when reading $w$ from right-to-left. As such, we have that, when reading any word in $\varphi_{a_i a_j}([w]_{\sylv})$ from right-to-left, the first occurrence of $1$ corresponds to the first occurrence of $a_i$, when reading $w$ from right-to-left, and all occurrences of $2$ before the first occurrence of $1$ correspond to all occurrences of $a_j$ before the first occurrence of $a_i$. Thus, $\varphi_{a_i a_j}([w]_{\sylv})$ has a $2$-$1$ right precedence if and only if $w$ has a $a_j$-$a_i$ right precedence, and the indexes must coincide. Hence, we get the following lemma:
	
	\begin{lemma}
    \label{lemma:varphi_equiv_sylvn}
		Let $u,v \in \anset^*$. Then, $u \equiv_{\sylv} v$ if and only if $\varphi_{ij}([u]_{\sylv}) = \varphi_{ij}([v]_{\sylv})$, for all $1 \leq i < j \leq n$.
	\end{lemma}
	
	\begin{proof}
	    The proof of the implication is trivial, since $\varphi_{ij}$ is well-defined as a map, for all $1 \leq i < j \leq n$. The proof of the converse follows from the previous observations, as well as Proposition~\ref{proposition:right_precedence_characterization}.
	\end{proof}

	For each $n \in \nset$, with $n \geq 3$, let $I_n$ be the index set
	\[
	    \left\{ (i,j): 1 \leq i < j \leq n \right\},
	\]
	and let $I := \bigcup_{n \in \nset} I_n$. Now, consider the map
	\[
	\phi_n : \sylvn \longrightarrow \prod\limits_{I_n} \sylv_2,
	\]
	whose $(i,j)$-th component is given by $\varphi_{ij}([w]_{\sylv})$, for $w \in \anset^*$ and $(i,j) \in I_n$.
		
	\begin{proposition}
    \label{prop:sylvn_embed_sylv2}
		The map $\phi_n$ is an embedding.
	\end{proposition}
	
	\begin{proof}
		It is clear that $\phi_n$ is a homomorphism. It follows from the definition of $\phi_n$ and Lemma~\ref{lemma:varphi_equiv_sylvn} that, for any $u,v \in \anset^*$, we have $u \equiv_{\sylvn} v$ if and only if $\phi_n([u]_{\sylv}) = \phi_n([v]_{\sylv})$, hence $\phi_n$ is an embedding.
	\end{proof}
	
	Thus, for each $n \in \nset$, we can embed $\sylvn$ into a direct product of copies of $\sylv_2$.
	
	Similarly, we can embed $\sylv$ into a direct product of infinitely many copies of $\sylv_2$. Consider the map
	\[
	\phi : \sylv \longrightarrow \prod\limits_{I} \sylv_2,
	\]
	whose $(i,j)$-th component is given by $\varphi_{ij}([w]_{\sylv})$, for $w \in \aset^*$ and $(i,j) \in I$.
	
	\begin{proposition}
    \label{prop:sylv_embed_sylv2}
		The map $\phi$ is an embedding.
	\end{proposition}
	
	\begin{proof}
		It is clear that $\phi$ is a homomorphism. Notice that, for any word $w \in \aset^*$, there must exist $n \in \nset$ such that $w \in \anset^*$. Furthermore, for $(i,j) \in I_n$, we have that the $(i,j)$-th component of $\phi([w]_{\sylv})$ is equal to the $(i,j)$-th component of $\phi_n([w]_{\sylv})$. Thus, for $u,v \in \aset^*$, if $\phi([u]_{\sylv}) = \phi([v]_{\sylv})$, then $\phi_n([u]_{\sylv}) = \phi_n([v]_{\sylv})$, for some $n \in \nset$ such that $u,v \in \anset^*$.
		
		It follows from Lemma~\ref{lemma:varphi_equiv_sylvn} that, for any $u,v \in \aset^*$, we have $u \equiv_{\sylv} v$ if and only if $\phi([u]_{\sylv}) = \phi([v]_{\sylv})$, hence $\phi$ is an embedding.
	\end{proof}
		
	As such, all sylvester monoids of rank higher than $2$ are in the variety generated by $\sylv_2$. Since $\sylv_2$ is a submonoid of $\sylv$ and $\sylvn$, for any $n \geq 3$, they all generate the same variety, which we will denote by $\mathbf{V}_{\sylv}$. Thus, by Birkhoff's Theorem, we have the following result:
	
	\begin{theorem}
    \label{theorem:sylv-id-sylv_2}
		For any $n \geq 2$, $\sylv$ and $\sylvn$ satisfy exactly the same identities.
	\end{theorem}
	
	Another consequence of $\mathbf{V}_{\sylv}$ being generated by $\sylv_2$ is the following:
	
	\begin{proposition}
	    The basis rank of $\mathbf{V}_{\sylv}$ is $2$.
	\end{proposition}
	
	\begin{proof}
	    Since $\mathbf{V}_{\sylv}$ is generated by $\sylv_2$, and $\sylv_2$ is defined by a presentation where the alphabet has two generators, then $r_b \left(\mathbf{V}_{\sylv}\right)$ is less than or equal to $2$.
	    
	    On the other hand, notice that any monoid generated by a single element is commutative. Since $\sylv$ is not commutative, $\mathbf{V}_{\sylv}$ cannot be generated by any single monoid which is itself generated by a single element. As such, $r_b \left(\mathbf{V}_{\sylv}\right)$ is strictly greater than $1$.
	    
	    Hence, the basis rank of $\mathbf{V}_{\sylv}$ is $2$.
	\end{proof}
		
	By parallel reasoning, we can also prove that the \#-sylvester monoids of rank greater than or equal to $2$ embed into a direct product of copies of $\sylvh_2$. For any $i,j \in \aset$, with $i < j$, define the homomorphism $\varphi^{\#}_{ij}: \aset^* \longrightarrow \sylvh_2$ in an identical fashion to $\varphi_{ij}$, mapping a word to a \#-sylvester class instead of a sylvester class.
	
	As with the case of $\varphi_{ij}$ (see Lemma~\ref{lemma:phi_factor_homomorphism}), we have that $\varphi^{\#}_{ij}$ factors to give a homomorphism $\varphi^{\#}_{ij} : \sylvh \longrightarrow \sylvh_2$. Furthermore, we can also deduce the number of occurrences of each symbol in a word $w \in \anset^*$, for some $n \geq 3$, and its left precedences, by looking at the images of $[w]_{\sylvh}$ under $\varphi^{\#}_{ij}$, ranging $1 \leq i < j \leq n$. Thus, for $u,v \in \anset^*$, we have that $u \equiv_{\sylvhn} v$ if and only if $\varphi^{\#}_{ij}([u]_{\sylvh}) = \varphi^{\#}_{ij}([v]_{\sylvh})$, for all $1 \leq i < j \leq n$.
    
	For each $n \in \nset$, we can embed $\sylvhn$ into a direct product of copies of $\sylvh_2$, using the embedding
	\[
	\phi^{\#}_n : \sylvhn \longrightarrow \prod\limits_{I_n} \sylvh_2,
	\]
	whose $(i,j)$-th component is given by $\varphi^{\#}_{ij}([w]_{\sylvh})$, for $w \in \anset^*$ and $(i,j) \in I_n$. Similarly, we can embed $\sylvh$ into a direct product of infinitely many copies of $\sylvh_2$, using the embedding
	\[
	\phi^{\#} : \sylvh \longrightarrow \prod\limits_{I} \sylvh_2,
	\]
	whose $(i,j)$-th component is given by $\varphi^{\#}_{ij}([w]_{\sylvh})$, for $w \in \aset^*$ and $(i,j) \in I$.

	As such, all \#-sylvester monoids of rank higher than $2$ are in the variety generated by $\sylvh_2$. Since $\sylvh_2$ is a submonoid of $\sylvh$ and $\sylvhn$, for any $n \geq 3$, they all generate the same variety, which we will denote by $\mathbf{V}_{\sylvh}$. Thus, by Birkhoff's Theorem, we have the following result:
	
	\begin{theorem}
    \label{theorem:sylvh-id-sylvh_2}
		For any $n \geq 2$, $\sylvh$ and $\sylvhn$ satisfy exactly the same identities.
	\end{theorem}
	
	The basis rank of $\mathbf{V}_{\sylvh}$ is also $2$.
	
	Notice that the embeddings of the \#-sylvester monoids of finite rank greater than or equal to $2$ into a direct product of copies of $\sylvh_2$ can also be obtained using the anti-isomorphisms between sylvester and \#-sylvester monoids of finite rank, and the the previously obtained embeddings. However, we cannot use this argument for the infinite rank case, due to Proposition~\ref{proposition:no_anti-isomorphism_sylv_sylvh}. On the other hand, since anti-isomorphisms exist in the finite case, we can conclude that, due to Theorems~\ref{theorem:sylv-id-sylv_2} and \ref{theorem:sylvh-id-sylvh_2}, any monoid anti-isomorphic to $\sylv$ (respectively, $\sylvh$) is in the variety generated by $\sylvh$ (respectively, $\sylv$).
    
    Once again, we can also prove that the Baxter monoids of rank greater than or equal to $2$ embed into a direct product of copies of $\baxt_2$. For any $i,j \in \aset$, with $i < j$, define the homomorphism $\vartheta_{ij}: \aset^* \longrightarrow \baxt_2$ in an identical fashion to $\varphi_{ij}$, mapping a word to a Baxter class instead of a sylvester class.
	
	As with the case of $\varphi_{ij}$, we have that $\vartheta_{ij}$ factors to give a homomorphism $\vartheta_{ij} : \baxt \longrightarrow \baxt_2$. Furthermore, we can also deduce the number of occurrences of each symbol in a word $w \in \anset^*$, for some $n \geq 3$, and its left and right precedences, by looking at the images of $[w]_{\baxt}$ under $\vartheta_{ij}$, ranging $1 \leq i < j \leq n$. Thus, for $u,v \in \anset^*$, we have that $u \equiv_{\baxtn} v$ if and only if $\vartheta_{ij}([u]_{\baxt}) = \vartheta_{ij}([v]_{\baxt})$, for all $1 \leq i < j \leq n$.
    
	For each $n \in \nset$, we can embed $\baxtn$ into a direct product of copies of $\baxt_2$, using the embedding
	\[
	\theta_n : \baxtn \longrightarrow \prod\limits_{I_n} \baxt_2,
	\]
	whose $(i,j)$-th component is given by $\vartheta_{ij}([w]_{\baxt})$, for $w \in \anset^*$ and $(i,j) \in I_n$. Similarly, we can embed $\baxt$ into a direct product of infinitely many copies of $\baxt_2$, using the embedding
	\[
	\theta : \baxt \longrightarrow \prod\limits_{I} \baxt_2,
	\]
	whose $(i,j)$-th component is given by $\vartheta_{ij}([w]_{\baxt})$, for $w \in \aset^*$ and $(i,j) \in I$.

	As such, all Baxter monoids of rank higher than $2$ are in the variety generated by $\baxt_2$. Since $\baxt_2$ is a submonoid of $\baxt$ and $\baxtn$, for any $n \geq 3$, they all generate the same variety, which we will denote by $\mathbf{V}_{\baxt}$. Thus, by Birkhoff's Theorem, we have the following result:
	
	\begin{theorem}
    \label{theorem:baxt-id-baxt_2}
		For any $n \geq 2$, $\baxt$ and $\baxtn$ satisfy exactly the same identities.
	\end{theorem}
	
	The basis rank of $\mathbf{V}_{\baxt}$ is also $2$.
	
	
	\section{Identities and bases}
	\label{section:identities_and_bases}
	
	
	In this section, we obtain a complete characterization of the identities satisfied, respectively, by the sylvester, \#-sylvester and Baxter monoids, finite bases for $\mathbf{V}_{\sylv}$, $\mathbf{V}_{\sylvh}$ and $\mathbf{V}_{\baxt}$, and also their axiomatic rank. 
	
	\subsection{Characterization of the identities satisfied by the sylvester, \#\chyph sylvester and Baxter monoids}
    \label{subsection:identities}
	
	The identities satisfied by the sylvester, \#-sylvester and Baxter monoids and those satisfied by their respective monoids of rank $2$ are exactly the same. As such, we shall use the monoids of rank $2$ to obtain a characterization of those identities.
	
	For a word $u$ over an alphabet of variables $X$, and for variables $x,y \in supp(u)$, we denote the number of occurrences of $y$ before the first occurrence of $x$ in $u$, when reading $u$ from right-to-left (respectively, from left-to-right), by $o_{x \leftarrow y} (u)$ (respectively, $o_{y \rightarrow x} (u)$).
			
	\begin{theorem}
    \label{theorem:sylv_identities}
		The identities $u \approx v$ satisfied by $\sylv$ are exactly the balanced identities such that, for any variables $x,y \in supp(u \approx v)$, $o_{x \leftarrow y} (u) = o_{x \leftarrow y} (v)$.
	\end{theorem}
	
	\begin{proof}		
		We first prove by contradiction that an identity satisfied by $\sylv_2$ must satisfy the previously mentioned conditions. Suppose $u \approx v$ is an identity satisfied by $\sylv_2$. Since $\sylv_2$ contains the free monogenic submonoid, we know that any identity satisfied by $\sylv_2$ must be a balanced identity. Thus, we assume $u \approx v$ is a balanced identity. 
		
		Suppose, in order to obtain a contradiction, that there exist variables $x,y \in supp(u \approx v)$, such that $o_{x \leftarrow y} (u) \neq o_{x \leftarrow y} (v)$. Then, if we consider the words $u|_{x,y}$ and $v|_{x,y}$, obtained from $u$ and $v$, respectively, by eliminating every occurrence of a variable other than $x$ or $y$, we have that $u|_{x,y}$ admits the suffix $x y^{o_{x \leftarrow y} (u)}$ and $v|_{x,y}$ admits the suffix $x y^{o_{x \leftarrow y} (v)}$.
		
		Taking the evaluation $\psi$ of $X$ in $\sylv_2$ such that $\psi (x) = [1]_{\sylv_2}$, $\psi (y) = [2]_{\sylv_2}$ and $\psi (z) = [\varepsilon]_{\sylv_2}$, for all other variables $z \in X$, we have
		\begin{align*}
    		\psi (u) = \psi \left( u|_{x,y} \right) = [u']_{\sylv_2} \cdot [1 2^{o_{x \leftarrow y} (u)}]_{\sylv_2} & \quad \text{and} \\
    		\psi (v) = \psi \left( v|_{x,y} \right) = [v']_{\sylv_2} \cdot [1 2^{o_{x \leftarrow y} (v)}]_{\sylv_2}, &
		\end{align*}
		for some words $u', v' \in \aset_2^*$. Since $o_{x \leftarrow y} (u) \neq o_{x \leftarrow y} (v)$, we have that $\psi (u)$ and $\psi (v)$ cannot share a $2$-$1$ right precedence of the same index. Thus, by Proposition~\ref{proposition:right_precedence_characterization}, we have that $\psi (u) \neq \psi (v)$, which contradicts our hypothesis that $u \approx v$ is an identity.
			
		We now prove by contradiction that an identity which satisfies the previously mentioned conditions must also be satisfied by $\sylv_2$. Suppose that $u \approx v$ is a balanced identity, such that, for any variables $x,y \in supp(u \approx v)$, $o_{x \leftarrow y} (u) = o_{x \leftarrow y} (v)$. Suppose, in order to obtain a contradiction, that there is some evaluation $\psi$ of $X$ in $\sylv_2$ such that $\psi(u) \neq \psi(v)$.
		
		Notice that, since $u \approx v$ is a balanced identity, then $\psi(u)$ and $\psi(v)$ have the same content. As such, we have that $supp \left( \psi(u) \right) = supp \left( \psi(v) \right) = \{ 1,2 \}$, and, by Proposition~\ref{proposition:right_precedence_characterization}, either $\psi(u)$ and $\psi(v)$ have $2$-$1$ right precedences, of different indexes, or one of them has a $2$-$1$ right precedence and the other does not. 
		Assume, without loss of generality, that words in $\psi(u)$ admit a suffix of the form $1 2^a$, and words in $\psi(v)$ admit a suffix of the form $1 2^b$, for some $a,b \in \nset_0$ such that $a > b$. Notice that this assumption covers both the case where $\psi(v)$ has a right precedence and the case where it has not. Furthermore, the assumption implies that $\psi(u)$ has a $2$-$1$ right precedence of index $a$.
		
		Observe that $u$ must be of the form $u = u_1 z u_2$, with $z \in X$ and $u_2 \in X^*$, such that $\psi (u_2)$ has support $\{2\}$ and $\psi (z)$ has either support $\{1,2\}$ or support $\{1\}$. As such, $\psi (z u_2)$ has a $2$-$1$ right precedence of index $a$, the same as $\psi(u)$. Furthermore, notice that $z$ cannot occur in $u_2$. Thus, by our hypothesis, we have that $o_{z \leftarrow x} (u) = o_{z \leftarrow x} (v)$, for any variable $x \in supp(u_2)$. Therefore, $v$ must also be of the form $v = v_1 z v_2$, where $v_2$ has the same content as $u_2$. But this implies that $\psi (v_2)$ has support $\{2\}$, hence $\psi(z v_2)$ also has a $2$-$1$ right precedence of index $a$. Since $\psi(z v_2)$ must also have the same right precedence as $\psi(v)$, we have reached a contradiction. 
		
		Thus, there is no evaluation $\psi$ of $X$ in $\sylv_2$ such that $\psi(u) \neq \psi(v)$. Therefore, $u \approx v$ is an identity satisfied by $\sylv_2$.
		
		Since all identities satisfied by $\sylv$ must also be satisfied by $\sylv_2$, we obtain the stated result.
	\end{proof}
	
	By parallel reasoning, we can obtain the characterization of the identities satisfied by the \#-sylvester monoid:
	
	\begin{theorem}
    \label{theorem:sylvh_identities}
		The identities $u \approx v$ satisfied by $\sylvh$ are exactly the balanced identities such that, for any variables $x,y \in supp(u \approx v)$, $o_{x \rightarrow y} (u) = o_{x \rightarrow y} (v)$.
	\end{theorem}
	
	From the previous two theorems, we can easily obtain the characterization of the identities satisfied by the Baxter monoid:
	
	\begin{theorem}
    \label{theorem:baxt_identities}
		The identities $u \approx v$ satisfied by $\baxt$ are exactly the balanced identities such that, for any variables $x,y \in supp(u \approx v)$, $o_{x \rightarrow y} (u) = o_{x \rightarrow y} (v)$ and $o_{x \leftarrow y} (u) = o_{x \leftarrow y} (v)$.
	\end{theorem}
	
	\begin{proof}
        Let $u \approx v$ be an identity and let $\psi$ be an evaluation of $X$ in $\baxt$. Let $u',v' \in \aset^*$ be words such that $\psi(u) = \left[ u' \right]_{\baxt}$ and $\psi(v) = \left[ v' \right]_{\baxt}$. 
        
        Notice that, by composing $\psi$ with the natural homomorphisms of $\baxt$ into $\sylv$ and $\sylvh$, we obtain evaluations $\psi_1$ and $\psi_2$ of $X$ into $\sylv$ and $\sylvh$, respectively. Furthermore, notice that $\psi_1 (w) = \left[ w' \right]_{\sylv}$ and $\psi_2 (w) = \left[ w' \right]_{\sylvh}$, for $w \in \{ u,v \}$.
        
        Hence, by {\cite[Proposition~3.7]{giraudo_baxter2}}, we have that $\psi (u) = \psi (v)$ if and only if $\psi_1 (u) = \psi_1 (v)$ and $\psi_2 (u) = \psi_2 (v)$. As such, if $u \approx v$ is satisfied by both $\sylv$ and $\sylvh$, then it must also be satisfied by $\baxt$. 
        
        On the other hand, any identity satisfied by $\baxt$ must be satisfied by both $\sylv$ and $\sylvh$, since both these monoids are in the variety generated by $\baxt$. Thus, the result follows as a consequence of Theorems \ref{theorem:sylv_identities} and \ref{theorem:sylvh_identities}.
	\end{proof}

    Recall that two identities are equivalent if one can be obtained from the other by renaming variables or swapping both sides of the identities. With these characterizations, we recover the following corollaries:
    
	\begin{corollary}[{\cite[Proposition~20]{cm_identities}}]
	\label{corollary:cm_identities_sylv}
		The sylvester monoid satisfies the non-trivial identity $xyxy \approx yxxy$. Furthermore, up to equivalence, this is the shortest non-trivial identity satisfied by $\sylv$.
	\end{corollary}
	
	\begin{corollary}[{\cite[Proposition~24]{cm_identities}}]
	\label{corollary:cm_identities_sylvh}
		The \#-sylvester monoid satisfies the non-trivial identity $yxyx \approx yxxy$. Furthermore, up to equivalence, this is the shortest non-trivial identity satisfied by $\sylvh$.
	\end{corollary}

	\begin{corollary}[{\cite[Proposition~26]{cm_identities}}]
	\label{corollary:cm_identities_baxt}
		The Baxter monoid satisfies the non-trivial identities $yxxyxy \approx yxyxxy$ and $xyxyxy \approx xyyxxy$. Furthermore, up to equivalence, these are the shortest non-trivial identities satisfied by $\baxt$.
	\end{corollary}
	
	The following corollaries are useful alternative characterizations of the identities satisfied by $\sylv$, $\sylvh$ and $\baxt$. They state that, when reading both sides of an identity satisfied by the sylvester, \#-sylvester or Baxter monoid, the first occurrence of a variable is read at the same time in both words:
    
    \begin{corollary}
	\label{corollary:sylv_h_first_occurrence}
        The identities $u \approx v$ satisfied by $\sylv$ (respectively, $\sylvh$) are balanced identities such that, for any $x \in supp(u \approx v)$, the longest suffix (respectively, prefix) of $u$ where $x$ does not occur has the same content as the longest suffix (respectively, prefix) of $v$ where $x$ does not occur.
	\end{corollary}
	
	\begin{proof}
	    We give the proof for the $\sylv$ case. The reasoning for the $\sylvh$ case is parallel.
	
		Let $u = u_1 x u_2$ and $v= v_1 x v_2$, where $u_2$ and $v_2$ are words where $x$ does not occur. Notice that, due to Theorem~\ref{theorem:sylv_identities}, for any variable $y$ which occurs in $u_2$ or $v_2$, we have that $o_{x \leftarrow y} (u) = o_{x \leftarrow y} (v)$. The result follows immediately.
	\end{proof}
	
	\begin{corollary}
	\label{corollary:baxt_first_occurrence}
        The identities $u \approx v$ satisfied by $\baxt$ are balanced identities such that, for any $x \in supp(u \approx v)$, the longest prefix of $u$ where $x$ does not occur has the same content as the longest prefix of $v$ where $x$ does not occur, and the longest suffix of $u$ where $x$ does not occur has the same content as the longest suffix of $v$ where $x$ does not occur.
	\end{corollary}

	\begin{proof}
		The result follows from the previous corollary.
	\end{proof}

    These alternate characterizations allow us to obtain algorithms which check if identities are satisfied by the sylvester, \#-sylvester and Baxter monoids in polynomial time. For brevity's sake, we only show the algorithm for the sylvester case:

    \begin{algorithm}
        \KwIn{An identity $u \approx v$.}     
        \KwOut{True if $\sylv$ satisfies $u \approx v$, False otherwise.}
        \BlankLine
        \lIf{$|u| \neq |v|$}{\Return False;}
        $k \leftarrow |u|$;\;
        $C[1,\dots,k],D[1,\dots,k] \leftarrow [0,\dots,0]$; \;
        $\overleftarrow{s} \leftarrow \emptyset$; \;
        \For{$0 \leq i \leq k-1$}{
            \If{$u_{k-i} = v_{k-i}$}{
                \If{$u_{k-i} \notin supp\left(\overleftarrow{s}\right)$}{
                    \If{$C \neq D$}{\Return False;}
                    \Else{
                        append $u_{k-i}$ to $\overleftarrow{s}$; \;
                        $j \leftarrow |\overleftarrow{s}|$; \;
                        $C[j] \leftarrow C[j] + 1$;
                        $D[j] \leftarrow D[j] + 1$;
                    }
                }
                \Else{
                    $j \leftarrow$ index of $u_{k-i}$ in $\overleftarrow{s}$; \;
                    $C[j] \leftarrow C[j] + 1$;
                    $D[j] \leftarrow D[j] + 1$;
                }
            }
            \Else{
                \If{$u_{k-i},v_{k-i} \in supp\left(\overleftarrow{s}\right)$}{
                    $j \leftarrow$ index of $u_{k-i}$ in $\overleftarrow{s}$; \;
                    $l \leftarrow$ index of $v_{k-i}$ in $\overleftarrow{s}$; \;
                    $C[j] \leftarrow C[j] + 1$;
                    $D[l] \leftarrow D[l] + 1$;
                }
                \Else{\Return False;}
            }
        }
        \lIf{$C \neq D$}{\Return False;}
        \Return True
        \caption{\textit{Identity checking algorithm for the sylvester monoid}.}
        \label{alg_sylv_id}
    \end{algorithm}
    
    \begin{proposition}
        Algorithm~\ref{alg_sylv_id} is sound and complete, and has time complexity $\mathcal{O} (k^2 \log(k))$, where $k$ is the length of the word $u$, for input $u \approx v$.
    \end{proposition}
    
    \begin{proof}
        Algorithm~\ref{alg_sylv_id} first checks if $u$ and $v$ have the same length, in line 1. If they do not, then $u \approx v$ is not a balanced identity, and as such, is not satisfied by $\sylv$. This is done in $2k+1$ time, in the worst-case scenario where the length of $v$ is greater than the length of $u$.
                
        The algorithm scans $u$ and $v$, from right-to-left, in the `\textbf{for}' cycle in line 5. The arrays $C$ and $D$ stand for, respectively, the content vectors of the suffixes of $u$ and $v$ read so far, while the word $\overleftarrow{s}$ stands for the support of these suffixes. Notice that, since $u$ is of length $k$, then at most $k$ variables occur in $u$. Hence, $C$ and $D$ have length $k$.
        
        In each loop of the cycle, the algorithm checks if the letter which is being read in $u$ is the same as the one being read in $v$. If they are the same, and they do not occur in $\overleftarrow{s}$, this means that this is the first occurrence of a variable $x$. The algorithm checks if the arrays $C$ and $D$ are equal. If they are not, this implies that the longest suffix of $u$ where $x$ does not occur does not have the same content as the longest suffix of $v$ where $x$ does not occur. Hence, by Corollary~\ref{corollary:sylv_h_first_occurrence}, $u \approx v$ is not satisfied by $\sylv$. If $C$ and $D$ are equal, then the algorithm registers the new variable in $\overleftarrow{s}$ and updates the content vectors $C$ and $D$. On the other hand, if $x$ occurs already in $\overleftarrow{s}$, the algorithm simply updates $C$ and $D$.
        
        If the letters which are being read in $u$ and $v$ are different, then the algorithm checks if they both occur in $\overleftarrow{s}$. If that does not happen, then that means at least one of them is the first occurrence of a variable in one of the words, but not in the other. Hence, by Corollary~\ref{corollary:sylv_h_first_occurrence}, $u \approx v$ is not satisfied by $\sylv$. Otherwise, if they both occur in $\overleftarrow{s}$, the algorithm simply updates $C$ and $D$.
    
        It is clear that the algorithm is sound and complete, since it always detects when a new variable is read, if it is read at the same time in both $u$ and $v$, and if the content of the suffixes is the same.
        
        Taking into consideration that operations of addition and comparing numbers are logarithmic time in a Turing machine model, and that accessing coordinates of vectors is a linear-time operation, we have that comparing the content vectors $C$ and $D$ takes at most $\mathcal{O} \left(k \log(k)\right)$ time, and updating them takes $\mathcal{O} \left(k \log(k)\right)$ time as well. On the other hand, checking if a variable occurs $\overleftarrow{s}$ takes $\mathcal{O} (k)$ time. As such, each loop of the \textbf{for} cycle has time complexity $\mathcal{O} \left(k \log(k)\right)$. Since there are $k$ loops of the cycle, and no other part of the algorithm takes as much time as the cycle, we can conclude that Algorithm~\ref{alg_sylv_id} has time complexity $\mathcal{O} \left(k^2 \log(k)\right)$.
    \end{proof}
    
    \begin{corollary}
        The decision problem $\textsc{Check-Id}(\sylv)$ belongs to the complexity class \textbf{P}.
    \end{corollary}
    
    We can also construct an algorithm that checks if identities hold in $\sylvh$, with time complexity $\mathcal{O} (k^2 \log(k))$. From that algorithm and Algorithm~\ref{alg_sylv_id}, we can construct another algorithm for the Baxter case. As such, we also have the following corollary:
    
    \begin{corollary}
        The decision problems $\textsc{Check-Id}(\sylvh)$ and $\textsc{Check-Id}(\baxt)$ belong to the complexity class \textbf{P}.
    \end{corollary}
    
	We also easily obtain some examples of important non-trivial identities satisfied by these monoids:
	
	\begin{example}
	\label{example:basis_identities}
		Consider the following non-trivial identities:
		\begin{align}
            xyzxty \approx yxzxty; \tag{L} \label{idL}\\
            xzxytx \approx xzyxtx; \tag{M} \label{idM}\\
            xzytxy \approx xzytyx. \tag{R} \label{idR}
        \end{align}
        
        The sylvester monoid satisfies \eqref{idL}, but satisfies neither \eqref{idM} nor \eqref{idR}, while the \#-sylvester monoid satisfies \eqref{idR}, but satisfies neither \eqref{idL} nor \eqref{idM}.
        
        The Baxter monoid satisfies the following non-trivial identities:
        \begin{align}
            xzyt \; xy \; rxsy \approx xzyt \; yx \; rxsy; \tag{O} \label{idO}\\
            xzyt \; xy \; rysx \approx xzyt \; yx \; rysx. \tag{E} \label{idE}
        \end{align}

	\end{example}
		
	The following corollaries will be important in the next subsection:
	
	\begin{corollary}
    \label{corollary:sylv_sylvh_shortest_identity_n_variables}
		The shortest non-trivial identities, with $n$ variables, satisfied by $\sylv$ or by $\sylvh$, are of length $n+2$.
	\end{corollary}
	
	\begin{proof}
		Since any identity satisfied by $\sylv$ must also be satisfied by $\hypo$, and since we already know that the shortest non-trivial identity, with $n$ variables, satisfied by $\hypo$, is of length $n+2$ (see {\cite[Corollary~4.6]{cmr_hypo_id}}), then a non-trivial identity, with $n$ variables, satisfied by $\sylv$, must be of length at least $n+2$.
		
		On the other hand, by Theorem~\ref{theorem:sylv_identities}, it is immediate that
		\[
		x y a_1 \dots a_{n-2} y x \approx y x a_1 \dots a_{n-2} y x
		\]
		is an identity satisfied by $\sylv$, for variables $x,y, a_1, \dots, a_{n-2}$.
		
		The reasoning for identities satisfied by $\sylvh$ is parallel to the one given previously.
	\end{proof}
	
	\begin{corollary}
    \label{corollary:baxt_shortest_identity_n_variables}
		The shortest non-trivial identity, with $n$ variables, satisfied by $\baxt$, is of length $n+4$.
	\end{corollary}

	\begin{proof}
		It is immediate, by Theorem~\ref{theorem:baxt_identities}, that for variables $x,y, a_1 \dots a_{n-2}$, 
		\[
		x y \; x y \; a_1 \dots a_{n-2} y x \approx x y \; y x \; a_1 \dots a_{n-2} y x
		\]
		is an identity satisfied by $\baxt$. 
		
		On the other hand, let $u \approx v$ be a non-trivial identity, with $n$ variables, satisfied by $\baxt$, such that $u = w x u'$ and $v = w y v'$, for some words $w, u', v'$ over the alphabet of variables $X$.
		
		Observe that $y$ must occur in $w$, otherwise, we would have $o_{x \rightarrow y} (u) > o_{x \rightarrow y} (v)$. Furthermore, it must also occur in $u'$, since $u \approx v$ is a balanced identity, hence $c(x u') = c(y v')$. Similarly, $x$ must occur in both $w$ and $v'$. On the other hand, by Corollary~\ref{corollary:baxt_first_occurrence}, $x$ must occur in $u'$ and $y$ must occur in $v'$, since $|u'| = |v'|$. Therefore, $x$ and $y$ both occur at least three times each in $u$ and $v$. Since $u \approx v$ is an identity with $n$ variables, it must be of length at least $n+4$.
	\end{proof}
	
	Finally, we can also clarify the relation between the Baxter monoids and the plactic monoids:
	
	\begin{corollary}
	    The variety generated by $\baxt$ is strictly contained in the variety generated by $\plac_2$.
	\end{corollary}
	
	\begin{proof}
	    Let $u \approx v$ be an identity satisfied by $\plac_2$. Thus, it must be a balanced identity. Let $x,y \in supp(u \approx v)$. Suppose, in order to obtain a contradiction, that $o_{x \rightarrow y} (u) > o_{x \rightarrow y} (v)$. Let
	    \[
	    u = u_1 \; y \; u_2 \quad \text{and} \quad v = v_1 \; y \; v_2,
	    \]
	    where $u_1$ (respectively, $v_1$) is the longest prefix of $u$ (respectively, $v$) where $y$ does not occur. Since the equational theory of the variety generated by $\plac_2$ is left 1-hereditary, then $u_1 \approx v_1$ must be satisfied by $\plac_2$. Hence, it must be a balanced identity. But $\left|u_1 \right|_x = o_{x \rightarrow y} (u) > o_{x \rightarrow y} (v) = \left|v_1 \right|_x$. We have reached a contradiction, hence, $o_{x \rightarrow y} (u) \not> o_{x \rightarrow y} (v)$.
	    
	    By this reasoning, we prove that $o_{x \rightarrow y} (u) = o_{x \rightarrow y} (v)$ and $o_{x \leftarrow y} (u) = o_{x \leftarrow y} (v)$. Hence, by Theorem~\ref{theorem:baxt_identities}, $u \approx v$ must be satisfied by $\baxt$. 
	    
	    On the other hand, it is well-known that the shortest non-trivial identity satisfied by the bicyclic monoid is Adjan's identity $xyyxxyxyyx \approx xyyxyxxyyx$ (see \cite{adjan_identity}). As such, $\plac_2$ does not satisfy any non-trivial identity of length less than $10$. But $\baxt$ satisfies an identity of length $6$, as seen in Corollary~\ref{corollary:cm_identities_baxt}. Thus, not all identities satisfied by $\baxt$ are satisfied by $\plac_2$.
	    
	    Therefore, as a consequence of Birkhoff's Theorem, the variety generated by $\baxt$ is strictly contained in the variety generated by $\plac_2$.	    
    \end{proof}
	
	
	\subsection{The axiomatic rank of the varieties generated by the sylvester, \#-sylvester and Baxter monoids}
    \label{subsection:basis_axiomatic_rank}
	
	Now, we prove that that the varieties generated by the sylvester, \#-sylvester and Baxter monoids have finite axiomatic rank and are finitely based. We give bases for $\mathbf{V}_{\sylv}$ and $\mathbf{V}_{\sylvh}$ with one identity each, of length $6$, over a four-letter alphabet. Trivially, these bases are minimal with regards to the number of identities in the basis; they are also minimal with regards to the number of variables occurring in these identities, and the length of these identities. We also give a basis for $\mathbf{V}_{\baxt}$, with two identities, of length $10$, over a six-letter alphabet. This basis is also minimal with regards to the number of identities in the basis, the number of variables occurring in these identities, and the length of these identities.
	
	Furthermore, we also prove that there exist no bases for $\mathbf{V}_{\sylv}$ or $\mathbf{V}_{\sylvh}$ with only identities over an alphabet with at most three variables, thus showing that the axiomatic rank of $\mathbf{V}_{\sylv}$ and $\mathbf{V}_{\sylvh}$ is $4$. We also prove that there exists no basis for $\mathbf{V}_{\baxt}$ with only identities over an alphabet with at most five variables, thus showing that the axiomatic rank of $\mathbf{V}_{\baxt}$ is $6$.
	
	The following theorem has been proven by different means in the upcoming paper \cite[Theorem~6.9]{kambites_johnson_cain_malheiro_preprint}.
	
	\begin{theorem}
	\label{theorem:sylv_finite_basis}
		$\mathbf{V}_{\sylv}$ admits a finite basis $\mathcal{B}_{\sylv}$, with the following identity:
		\begin{align}
            xyzxty \approx yxzxty. \tag{L}
        \end{align}
	\end{theorem}
	
	\begin{proof}
		Let $\mathcal{B}_{\sylv}$ be the set of identities which contains only the identity~\eqref{idL}. Notice that this identity is given in Example \ref{example:basis_identities}.
		
		The following proof will be done by induction, in the following sense: We order identities by the length of the common suffix of both sides of the identity. The induction will be on the length of the prefix, that is, the length of the identity minus the length of the common suffix. 
		
		The base cases for the induction, for identities of length $n$ (with $n \geq 4$), are those identities of the form 
		\[
		x y w \approx y x w,
		\]
		where $w$ is a word of length $n-2$ and $x,y$ are variables. Observe that, since any identity $u \approx v$ satisfied by $\sylv$ is a balanced identity, there are no non-trivial identities, of length $n$, with a common prefix of length greater than $n-2$, satisfied by $\sylv$. Furthermore, $x$ and $y$ must both occur in $w$, otherwise, we would have $o_{x \leftarrow y} (xyw) > o_{x \leftarrow y} (yxw)$. Thus, $w$ is of the form
		\[
		w_1 x w_2 y w_3 \quad \text{or} \quad w_1 y w_2 x w_3,
		\]
		for some words $w_1, w_2, w_3$. Therefore, by replacing $z$ with $w_1$, and $t$ by $w_2$, and, if necessary, renaming $x$ and $y$, we can immediately deduce this identity from the identity \eqref{idL}. Notice that, when $n=4$, the base cases correspond to the identities given in Corollary~\ref{corollary:cm_identities_sylv}.
		
		The idea of the proof of the induction step is that, for any identity $u \approx v$, we can apply the identity \eqref{idL}, finitely many times, to deduce a new identity $u \approx u^*$ from $u \approx v$, such that $u^*$ is ``closer'' to $v$ than $u$, in the sense that $u^*$ and $v$ have a common suffix which is strictly longer than the common suffix of $u$ and $v$. Notice that $u^* \approx v$ is a consequence of $\mathcal{B}_{\sylv}$, by the induction hypothesis. As such, we can conclude that $u \approx v$ is a consequence of $\mathcal{B}_{\sylv}$.
		
		The technical part of the proof allows us to show that there is always a way to shuffle some variables of $u$ in such a way that we obtain $u^*$. We show that these variables must occur several times in $u$, thus allowing us to apply the identity \eqref{idL} to shuffle $u$ and obtain $u^*$.
		
		By the induction hypothesis, we know that $u^* \approx v$ is a consequence of $\mathcal{B}_{\sylv}$. Thus, we conclude that $u \approx v$ is also a consequence of $\mathcal{B}_{\sylv}$. 
		
		Let $u \approx v$ be a non-trivial identity satisfied by $\sylv$ and let $\mathcal{X}$ denote its support. Since $u \approx v$ is a non-trivial identity, we must have $u = u' x w$ and $v = v' y w$, for some words $w,u',v' \in \mathcal{X}^*$ and variables $x,y$ such that $x \neq y$. Notice that $u'$ and $v'$ cannot be the empty word, otherwise, we would have $u= x w$ and $v = y w$, which contradicts the fact that $c(u) = c(v)$.
		
		On the other hand, since $c(u' x) = c(v' y)$, we have that $y$ occurs in $u'$. Thus, to distinguish the rightmost $y$ in $u'$, we have that
		\[
		u'x = u_1 y a u_2,
		\]
		for some variable $a$ and words $u_1$ and $u_2$, such that $y$ does not occur in $u_2$ and $a \neq y$. Once again, since $c(u' x) = c(v' y)$, we have that $a$ occurs in $v'$. Thus, to distinguish the rightmost $a$ in $v'$, we have that 
		\[
		v' = v_1 a v_2,
		\]
		for some words $v_1$ and $v_2$, such that $a$ does not occur in $v_2$. To sum up, we have that 
		\[
		u = u_1 y a u_2 w \quad \text{and} \quad v = v_1 a v_2 y w,
		\]
		where $y$ does not occur in $u_2$ and $a$ does not occur in $v_2$.
		
		Notice that, by Theorem~\ref{theorem:sylv_identities}, $y$ must occur in $w$, otherwise, we would have $o_{y \leftarrow x} (u) > o_{y \leftarrow x} (v)$. Thus, $a$ must also occur in $w$, otherwise, we would have $o_{a \leftarrow y} (u) < o_{a \leftarrow y} (v)$. As such, we can deduce the word $u_1 a y u_2 w$, by applying the identity \eqref{idL}, renaming $x$ to $a$ and replacing $z$ and $t$ by the appropriate words.
		
		Observe that we can repeatedly apply this reasoning until we obtain a word of the form
		\[
		u^* = u'' y w,
		\]
		for some word $u''$, since the only restriction imposed on the variable $a$ was that $a \neq y$. Thus, we have proven that, for any non-trivial identity $u \approx v$ satisfied by $\sylv$, we can obtain a new word $u^*$ from $u$ such that the common suffix of $u^*$ and $v$ is strictly longer than the common suffix of $u$ and $v$, by applying the identity \eqref{idL} finitely many times. 
		By the induction method, we conclude that $u \approx v$ is a consequence of $\mathcal{B}_{\sylv}$, thus proving that $\mathcal{B}_{\sylv}$ is a basis for $\mathbf{V}_{\sylv}$.
	\end{proof}
		
	By the same reasoning, we can also prove the following result:
	
	\begin{theorem}
	\label{theorem:sylvh_finite_basis}
		$\mathbf{V}_{\sylvh}$ admits a finite basis $\mathcal{B}_{\sylvh}$, with the following identity:
		\begin{align}
            xzytxy \approx xzytyx. \tag{R}
        \end{align}
	\end{theorem}
	
	\begin{proof}
	    The proof follows the same reasoning as the proof of Theorem~\ref{theorem:sylv_finite_basis}, the main difference being that, within a set of identities of the same length, they are ordered on the length of the common prefix of both sides of the identity. Therefore, the induction is on the length of the suffix. The induction step resorts to Theorem~\ref{theorem:sylvh_identities}.
	\end{proof}

	We also use the same reasoning to prove the following theorem:
	
	\begin{theorem}
	\label{theorem:baxt_finite_basis}
		$\mathbf{V}_{\baxt}$ admits a finite basis $\mathcal{B}_{\baxt}$, with the following identities:
        \begin{align}
            xzyt \; xy \; rxsy \approx xzyt \; yx \; rxsy; \tag{O} \\
            xzyt \; xy \; rysx \approx xzyt \; yx \; rysx. \tag{E}
        \end{align}
	\end{theorem}

	\begin{proof}
        The proof follows the same reasoning as the proof of Theorem~\ref{theorem:sylv_finite_basis}. As such, we only give the reasoning for the base cases and the induction step. 
        
        The base cases for the induction on the length of the prefix, for identities of length $n$ (with $n \geq 6$), are those identities of the form 
		\[
		x y x y w \approx x y y x w,
		\]
		where $w$ is a word of length $n-4$ and $x,y$ are variables. Notice that both sides of the identity must have a prefix of the form $xy$, due to Corollary~\ref{corollary:baxt_first_occurrence}. By the same reason, observe that, since any identity $u \approx v$ satisfied by $\baxt$ is a balanced identity, there are no non-trivial identities, of length $n$, with a common prefix of length greater than $n-4$, satisfied by $\baxt$. Furthermore, $x$ and $y$ must both occur in $w$, otherwise, we would have $o_{x \leftarrow y} (xyw) > o_{x \leftarrow y} (yxw)$. Thus, $w$ is of the form
		\[
		w_1 x w_2 y w_3 \quad \text{or} \quad w_1 y w_2 x w_3,
		\]
		for some words $w_1, w_2, w_3$. Therefore, by replacing $z$ and $t$ with the empty word, $r$ with $w_1$, and $s$ by $w_2$, and, if necessary, renaming $x$ and $y$, we can immediately deduce this identity from the identity \eqref{idO} or the identity \eqref{idE}, depending on the form of $w$. Notice that, when $n=6$, the base cases correspond to the identities given in Corollary~\ref{corollary:cm_identities_baxt}.

		Let $u \approx v$ be a non-trivial identity satisfied by $\baxt$ and let $\mathcal{X}$ denote its support. Since $u \approx v$ is a non-trivial identity, we must have $u = u' x w$ and $v = v' y w$, for some words $w,u',v' \in \mathcal{X}^*$ and variables $x,y$ such that $x \neq y$. By Corollary~\ref{corollary:baxt_first_occurrence}, we have that $x$ must occur at least once in $u'$ and $w$ and at least twice in $v'$, and $y$ must occur at least twice in $u'$ and at least once in $v'$ and $w$. Thus, to distinguish the leftmost $y$ in $u'$, we have that
		\[
		u'x = u_1 y a u_2
		\]
		for some variable $a$ and words $u_1$ and $u_2$, such that $y$ does not occur in $u_2$ and $a \neq y$. Notice that $y$ must occur in $u_1$. Since $c(u' x) = c(v' y)$, we have that $a$ occurs in $v'$. Thus, to distinguish the rightmost $a$ in $v'$, we have that 
		\[
		v' = v_1 a v_2,
		\]
		for some words $v_1$ and $v_2$, such that $a$ does not occur in $v_2$. To sum up, we have that 
		\[
		u = u_1 y a u_2 w \quad \text{and} \quad v = v_1 a v_2 y w,
		\]
		where $y$ occurs in $u_1$ but not in $u_2$ and $a$ does not occur in $v_2$.
		
		Suppose, in order to obtain a contradiction, that $a$ does not occur in $u_1$. This implies that $|u'|_y = o_{y \rightarrow a} (u)$. But $|u'|_y = |v'|_y + 1$, hence
		\[
		o_{y \rightarrow a} (v) \leq |v'|_y < |u'|_y = o_{y \rightarrow a} (u).
		\]
		Thus, by Theorem~\ref{theorem:baxt_identities}, we obtain a contradiction. As such, $a$ must occur in $u_1$. By the same theorem, $a$ must occur in $w$ as well, otherwise, we would have $o_{a \leftarrow y} (u) < o_{a \leftarrow y} (v)$. Therefore, $y$ and $a$ both occur at least once in $u_1$ and $w$. As such, we can deduce the word $u_1 a y u_2 w$, by applying the identity \eqref{idO} or the identity \eqref{idE}, depending on where $y$ and $a$ occur in $u_1$ and $w$, renaming $x$ to $a$ and replacing the remaining variables by the appropriate words.		
	\end{proof}
	
	An immediate consequence of having a finite basis is the following:
	
	\begin{corollary}
		The varieties $\mathbf{V}_{\sylv}$, $\mathbf{V}_{\sylvh}$ and $\mathbf{V}_{\baxt}$ have finite axiomatic rank.
	\end{corollary}
		
	By {\cite[Proposition~4.10]{cmr_hypo_id}} and {\cite[Proposition~4.11]{cmr_hypo_id}}, we know that the identities \eqref{idL} and \eqref{idR} are not consequences of the set of non-trivial identities, satisfied by $\hypo$, over an alphabet with four variables, excluding themselves and equivalent identities. Since the identities satisfied by $\sylv$ and $\sylvh$ must also be satisfied by $\hypo$, we can conclude the following:
	
	\begin{corollary}
        The identity \eqref{idL} is not a consequence of the set of non-trivial identities, satisfied by $\sylv$, over an alphabet with four variables, excluding \eqref{idL} itself and equivalent identities. Furthermore, any basis for $\mathbf{V}_{\sylv}$ with only identities over an alphabet with four variables must contain the identity \eqref{idL}, or an equivalent identity.
	\end{corollary}
	
	\begin{corollary}
        The identity \eqref{idR} is not a consequence of the set of non-trivial identities, satisfied by $\sylvh$, over an alphabet with four variables, excluding \eqref{idR} itself and equivalent identities. Furthermore, any basis for $\mathbf{V}_{\sylvh}$ with only identities over an alphabet with four variables must contain the identity \eqref{idR}, or an equivalent identity.
	\end{corollary}

	Hence, $\mathbf{V}_{\sylv}$ and $\mathbf{V}_{\sylvh}$ do not admit any bases with only identities over an alphabet with two or three variables. In other words, we have that:
	
	\begin{corollary}
        The axiomatic rank of $\mathbf{V}_{\sylv}$ and $\mathbf{V}_{\sylvh}$ is $4$.
	\end{corollary}
	
	We now show that the identities \eqref{idO} and \eqref{idE} are required to be in any basis for $\mathbf{V}_{\baxt}$ which contains only identities over an alphabet with six variables:
		
	\begin{proposition}
	\label{proposition:idO_E_not_consequence}
        The identities \eqref{idO} and \eqref{idE} are not consequences of the set of non-trivial identities, satisfied by $\baxt$, over an alphabet with six variables, excluding themselves (but not the other) and equivalent identities.
	\end{proposition}
	
	\begin{proof}
	    We prove the result for the identity \eqref{idO}. Parallel reasoning shows the analogous result for \eqref{idE}.
	    
        Let $X := \{ x,y,z,t,r,s \}$ and let $\mathcal{S}$ be the set of all non-trivial identities, satisfied by $\baxt$, over an alphabet with six variables, excluding \eqref{idO} and equivalent identities. Suppose, in order to obtain a contradiction, that \eqref{idO} is a consequence of $\mathcal{S}$. As such, there must exist a non-trivial identity $u \approx v$ in $\mathcal{S}$, and a substitution $\sigma$, such that
        \[
            xzyt \; xy \; rxsy = w_1 \sigma(u) w_2,
		\]
		where $w_1, w_2$ are words over $X$, and $\sigma(u) \neq \sigma(v)$. Notice that $u \approx v$ must be balanced, and that there must be at least two variables occurring in $u$ and $v$, otherwise, $u \approx v$ would be a trivial identity. 
				
        By the same reasoning as in the proof of \cite[Proposition~4.10]{cmr_hypo_id}, we can assume, without loss of generality, that $\sigma$ does not map any variable to the empty word. Due to this, and since only $x$ and $y$ occur three times in $xzyt \; xy \; rxsy$, and all other variables each occur one time, we have that each variable occurring in $u \approx v$ can occur at most three times, and only two variables can occur more than one time. Furthermore, by Corollary \ref{corollary:baxt_shortest_identity_n_variables}, which gives us a lower bound for the length of the identities, we have that $u \approx v$ is of length at least $6$. Notice that it is exactly of length $6$ if only two variables occur in it. Thus, up to renaming of variables, $x$ and $y$ occur exactly thrice in $u \approx v$, and $t$, $z$, $r$ and $s$ can occur at most one time.
                
        Suppose now, in order to obtain a contradiction, that $w_1 \neq \varepsilon$. Then, since $u \approx v$ is of length at least $6$, we must have $w_1$ of length at most $4$, that is, $w_1$ is either $x$, $xz$, $xzy$ or $xzyt$. Therefore, $x$ can occur only twice in $\sigma(u)$. But $x$ and $y$ occur thrice in $u$, and $\sigma$ does not map any variable to the empty word, hence, there must be at least two variables which occur thrice in $\sigma(u)$. However, only $x$ and $y$ occur thrice in $xzyt \; xy \; rxsy$. We have reached a contradiction, hence, $w_1 = \varepsilon$. Using a similar argument, we can also conclude that $w_2 = \varepsilon$. Therefore, we have that
        \[
            xzyt \; xy \; rxsy = \sigma(u).
		\]

        As such, we can immediately conclude that only up to five variables occur in $u \approx v$: If $u \approx v$ were to be a six-variable identity, then it would be of length $10$, and $\sigma$ would be simply renaming the variables, thus implying that $u \approx v$ was equivalent to \eqref{idO}, which contradicts our hypothesis.    
        
        Notice that, regardless of the number of variables occurring in $u \approx v$, we have that both $\sigma(x)$ and $\sigma(y)$ are a single variable, otherwise, more than two variables would have to occur three times in $xzyt \; xy \; rxsy$, or one variable would have to occur six times. Furthermore, $\sigma(x)$ and $\sigma(y)$ can only be $x$ or $y$, since these are the only variables occurring thrice in $xzyt \; xy \; rxsy$. Hence, if $u \approx v$ is an identity where up to five variables occur, then $u \approx v$ cannot be a two-variable identity, and, furthermore, there is at least one variable $z$ occurring in $u \approx v$ such that $\sigma(z)$ is of length at least $2$, and neither $x$ nor $y$ can occur in $\sigma(z)$. This is impossible, since $x$ or $y$ occur in every factor of $xzyt \; xy \; rxsy$ of length $2$.
        
        As such, we can conclude that \eqref{idO} is not a consequence of the set of non-trivial identities, satisfied by $\baxt$, over an alphabet with six variables, excluding \eqref{idO} itself and equivalent identities.
	\end{proof}
	    
	Therefore, we can conclude that $\mathbf{V}_{\baxt}$ does not admit any basis with only identities over an alphabet with up to five variables. In other words, we have that:
	
	\begin{corollary}
        The axiomatic rank of $\mathbf{V}_{\baxt}$ is $6$. Furthermore, any basis for $\mathbf{V}_{\baxt}$ with only identities over an alphabet with six variables must contain the identities \eqref{idO} and \eqref{idE}, or equivalent identities.
	\end{corollary}

	\bibliographystyle{alpha}
	\bibliography{sylv_baxt}

\newcommand{\etalchar}[1]{$^{#1}$}
\begin{thebibliography}{CKK{\etalchar{+}}17}

\bibitem[{Adj}67]{adjan_identity}
Sergei~I. {Adjan}.
\newblock ’\uppercase{D}efining relations and algorithmic problems for groups
  and semigroups’, \uppercase{T}r. \uppercase{M}at. \uppercase{I}nst.
  \uppercase{S}teklov 85.
\newblock {\em Am. Math. Soc.(translation)}, 152, 1967.

\bibitem[AU92]{aho_ullman_foundations}
Alfred~V. Aho and Jeffrey~D. Ullman.
\newblock {\em Foundations of Computer Science}.
\newblock Computer Science Press, Inc., USA, 1992.

\bibitem[Bax64]{baxter_glen}
Glen Baxter.
\newblock On fixed points of the composite of commuting functions.
\newblock {\em Proc. Amer. Math. Soc.}, 15:851--855, 1964.

\bibitem[Ber11]{Bergman_universal_algebra}
Clifford Bergman.
\newblock {\em Universal Algebra: Fundamentals and Selected Topics}.
\newblock Chapman \& Hall Pure and Applied Mathematics. Taylor \& Francis,
  2011.

\bibitem[BS81]{bs_universal_algebra}
S.~Burris and H.P. Sankappanavar.
\newblock {\em A Course in Universal Algebra}.
\newblock Graduate Texts in Mathematics. Springer New York, 1981.

\bibitem[BS17]{bump_crystalbases}
Daniel Bump and Anne Schilling.
\newblock {\em Crystal Bases: Representations and Combinatorics}.
\newblock World Scientific Publishing Company Pte. Limited, 2017.

\bibitem[CHK{\etalchar{+}}20]{volkov_kauffman3}
Yuzhu Chen, Xun Hu, Nikita~V. Kitov, Yanfeng Luo, and Mikhail~V. Volkov.
\newblock Identities of the kauffman monoid k3.
\newblock {\em Communications in Algebra}, 48(5):1956--1968, 2020.

\bibitem[CJKM21]{kambites_johnson_cain_malheiro_preprint}
Alan Cain, Marianne Johnson, Mark Kambites, and Ant{\'{o}}nio Malheiro.
\newblock Representations and identities of plactic-like monoids.
\newblock unpublished, 2021.

\bibitem[CKK{\etalchar{+}}17]{ckkmo_placticidentity}
{Alan J.} Cain, Georg Klein, {\L}ukasz Kubat, Ant\'{o}nio Malheiro, and Jan.
  Okni{\'n}ski.
\newblock A note on identities in plactic monoids and monoids of
  upper-triangular tropical matrices.
\newblock {\em ArXiv e-prints}, 2017.

\bibitem[CM17]{cm_hypo_crystal}
Alan~J. Cain and Ant\'onio Malheiro.
\newblock Crystallizing the hypoplactic monoid: from
  quasi-\uppercase{K}ashiwara operators to the
  \uppercase{R}obinson--\uppercase{S}chensted-type correspondence for
  quasi-ribbon tableaux.
\newblock {\em Journal of Algebraic Combinatorics}, 45:475--524, 2017.

\bibitem[CM18a]{cm_sylv_crystal}
Alan~J. Cain and António Malheiro.
\newblock Crystals and trees: Quasi-kashiwara operators, monoids of binary
  trees, and robinson–schensted-type correspondences.
\newblock {\em Journal of Algebra}, 502:347 -- 381, 2018.

\bibitem[CM18b]{cm_identities}
{Alan J.} Cain and Ant{\'{o}}nio Malheiro.
\newblock Identities in plactic, hypoplactic, sylvester, \uppercase{B}axter,
  and related monoids.
\newblock {\em Electronic Journal Of Combinatorics}, 25(3), 8 2018.

\bibitem[CM19]{cm_cyclic_shifts}
A.J. Cain and A.~Malheiro.
\newblock Combinatorics of cyclic shifts in plactic, hypoplactic, sylvester,
  baxter, and related monoids.
\newblock {\em Journal of Algebra}, 535:159--224, 2019.

\bibitem[CMR20]{cmr_hypo_id}
Alan~J. {Cain}, Ant{\'o}nio {Malheiro}, and Duarte {Ribeiro}.
\newblock {Identities and bases in the hypoplactic monoid}.
\newblock {\em arXiv e-prints}, page arXiv:2010.06953, October 2020.

\bibitem[DHT02]{duchamp2002noncommutative}
G{\'e}rard Duchamp, Florent Hivert, and Jean-Yves Thibon.
\newblock Noncommutative symmetric functions vi: free quasi-symmetric functions
  and related algebras.
\newblock {\em International Journal of Algebra and computation},
  12(05):671--717, 2002.

\bibitem[DJK18]{daviaud_identities}
Laure Daviaud, Marianne Johnson, and Mark Kambites.
\newblock Identities in upper triangular tropical matrix semigroups and the
  bicyclic monoid.
\newblock {\em Journal of Algebra}, 501:503 -- 525, 2018.

\bibitem[Ful96]{fulton_young}
William Fulton.
\newblock {\em Young Tableaux: With Applications to Representation Theory and
  Geometry}.
\newblock London Mathematical Society Student Texts. Cambridge University
  Press, 1996.

\bibitem[Gir12]{giraudo_baxter2}
Samuele Giraudo.
\newblock Algebraic and combinatorial structures on pairs of twin binary trees.
\newblock {\em Journal of Algebra}, 360:115 -- 157, 2012.

\bibitem[GKL{\etalchar{+}}94]{gelfandetal_noncommutative_symmetric_functions}
Israel Gelfand, Daniel Krob, Alain Lascoux, Bernard Leclerc, Vladimir Retakh,
  and J.~Thibon.
\newblock Noncommutative symmetric functions.
\newblock {\em Adv. Math.}, 112, 08 1994.

\bibitem[Gre06]{green2006polynomial}
James~A. Green.
\newblock {\em Polynomial Representations of GL\_n: With an Appendix on
  Schensted Correspondence and Littelmann Paths}, volume 830.
\newblock Springer, 2006.

\bibitem[Hig92]{higgins1992techniques}
Peter~M. Higgins.
\newblock {\em Techniques of semigroup theory}.
\newblock Oxford University Press on Demand, 1992.

\bibitem[HNT05]{hivert_sylvester}
Florent Hivert, Jean-Christophe Novelli, and Jean-Yves Thibon.
\newblock The algebra of binary search trees.
\newblock {\em Theoretical Computer Science}, 339(1):129 -- 165, 2005.
\newblock Combinatorics on Words.

\bibitem[How95]{howie1995fundamentals}
John~M. Howie.
\newblock {\em Fundamentals of Semigroup Theory}.
\newblock LMS monographs. Clarendon Press, 1995.

\bibitem[Izh19]{izhakian_tropical}
Zur Izhakian.
\newblock Tropical plactic algebra, the cloaktic monoid, and semigroup
  representations.
\newblock {\em Journal of Algebra}, 524:290 -- 366, 2019.

\bibitem[JK21]{johnson_kambites_tropical_plactic}
Marianne Johnson and Mark Kambites.
\newblock Tropical representations and identities of plactic monoids.
\newblock {\em Transactions of the American Mathematical Society},
  374:4423--4447, 2021.

\bibitem[JT19]{johnson_tran}
Marianne Johnson and Ngoc~Mai Tran.
\newblock Geometry and algorithms for upper triangular tropical matrix
  identities.
\newblock {\em Journal of Algebra}, 530:470--507, 2019.

\bibitem[Knu70]{knuth1970}
Donald~E. Knuth.
\newblock Permutations, matrices, and generalized \uppercase{Y}oung tableaux.
\newblock {\em Pacific J. Math.}, 34(3):709--727, 1970.

\bibitem[KO15]{kubat_identities}
{\L }ukasz Kubat and Jan Okni{\'{n}}ski.
\newblock {I}dentities of the plactic monoid.
\newblock {\em Semigroup Forum}, 90(1):100--112, February 2015.

\bibitem[KS95]{kharlampovich_sapir_survey}
O.~G. KHARLAMPOVICH and M.~V. SAPIR.
\newblock Algorithmic problems in varieties.
\newblock {\em International Journal of Algebra and Computation},
  05(04n05):379--602, 1995.

\bibitem[KT97]{Krob1997}
Daniel Krob and Jean-Yves Thibon.
\newblock Noncommutative symmetric functions \uppercase{IV}: Quantum linear
  groups and \uppercase{H}ecke algebras at $q=0$.
\newblock {\em Journal of Algebraic Combinatorics}, 6(4):339--376, October
  1997.

\bibitem[KV20]{volkov_kauffman4}
{Nikita V.} Kitov and {Mikhail V.} Volkov.
\newblock {\em Identities of the kauffman monoid K4 and of the Jones Monoid
  J4}, pages 156--178.
\newblock Lecture Notes in Computer Science (including subseries Lecture Notes
  in Artificial Intelligence and Lecture Notes in Bioinformatics). Springer
  Verlag, Germany, January 2020.

\bibitem[Lot02]{lothaire_2002}
M.~Lothaire.
\newblock {\em Algebraic Combinatorics on Words}.
\newblock Encyclopedia of Mathematics and its Applications. Cambridge
  University Press, 2002.

\bibitem[LR98]{loday_ronco_hopf_planar_binary_trees}
Jean-Louis Loday and María~O. Ronco.
\newblock Hopf algebra of the planar binary trees.
\newblock {\em Advances in Mathematics}, 139(2):293--309, 1998.

\bibitem[LS78a]{LS1978}
Alain Lascoux and Marcel-Paul Schützenberger.
\newblock Le monoïde plaxique.
\newblock {\em Noncommutative Structures in Algebra and Geometric
  Combinatorics, Naples}, Quad. Ricerca Sci., vol. 109:129--156, 1978.

\bibitem[LS78b]{lascoux_foulkes}
Alain Lascoux and Marcel-Paul Sch\"{u}tzenberger.
\newblock Sur une conjecture de {H}. {O}. {F}oulkes.
\newblock {\em C. R. Acad. Sci. Paris S\'{e}r. A-B}, 286(7):A323--A324, 1978.

\bibitem[Mac98]{macdonald_symmetric}
Ian~G. Macdonald.
\newblock {\em Symmetric Functions and Hall Polynomials}.
\newblock Oxford Classic Texts in the Physical Sciences. Clarendon Press, 1998.

\bibitem[Mal12]{mal2012algebraic}
Anatoly~I. Mal'cev.
\newblock {\em Algebraic Systems}.
\newblock Grundlehren der mathematischen Wissenschaften. Springer Berlin
  Heidelberg, 2012.

\bibitem[Mas96]{mashevitzky_hereditary}
G.~I. Mashevitzky.
\newblock On the finite basis problem for left hereditary systems of semigroup
  identities.
\newblock In J.~Almeida, P.~V. Silva, and G.~M.~S. Gomes, editors, {\em
  Semigroups, Automata and Languages}, pages 167--181. World Scientific, March
  1996.

\bibitem[MMT18]{mckenzie2018algebras}
Ralph~N. McKenzie, George~F. McNulty, and Walter~F. Taylor.
\newblock {\em Algebras, Lattices, Varieties}.
\newblock AMS Chelsea Publishing. American Mathematical Society, 2018.

\bibitem[MS15]{maclagan_tropical}
Diane Maclagan and Bernd Sturmfels.
\newblock {\em {I}ntroduction to {T}ropical {G}eometry}.
\newblock Number 161 in Graduate Studies in Mathematics. American Mathematical
  Society, 2015.

\bibitem[Nov00]{novelli_hypoplactic}
Jean-Christophe Novelli.
\newblock On the hypoplactic monoid.
\newblock {\em Discrete Mathematics}, 217(1):315 -- 336, 2000.

\bibitem[Pap94]{papadimitriou1994computational}
C.H. Papadimitriou.
\newblock {\em Computational Complexity}.
\newblock Theoretical computer science. Addison-Wesley, 1994.

\bibitem[Pas06]{pastijn_polyhedral}
Francis Pastijn.
\newblock Polyhedral convex cones and the equational theory of the bicyclic
  semigroup.
\newblock {\em Journal of the Australian Mathematical Society}, 81(1):63–96,
  2006.

\bibitem[PR95]{poirier_reutenauer}
St\'{e}phane Poirier and Christophe Reutenauer.
\newblock Alg\`ebres de {H}opf de tableaux.
\newblock {\em Ann. Sci. Math. Qu\'{e}bec}, 19(1):79--90, 1995.

\bibitem[Pri13]{priez_binary_trees}
Jean-Baptiste Priez.
\newblock {Lattice of combinatorial Hopf algebras: binary trees with
  multiplicities}.
\newblock In Alain Goupil and Gilles Schaeffer, editors, {\em {25th
  International Conference on Formal Power Series and Algebraic Combinatorics
  (FPSAC 2013)}}, volume DMTCS Proceedings vol. AS, 25th International
  Conference on Formal Power Series and Algebraic Combinatorics (FPSAC 2013) of
  {\em DMTCS Proceedings}, pages 1137--1148, Paris, France, 2013. {Discrete
  Mathematics and Theoretical Computer Science}.

\bibitem[Rea05]{reading_lattice_hopf}
Nathan Reading.
\newblock Lattice congruences, fans and {H}opf algebras.
\newblock {\em J. Combin. Theory Ser. A}, 110(2):237--273, 2005.

\bibitem[Sap14]{sapir_combinatorial}
Mark~V. Sapir.
\newblock {\em Combinatorial Algebra: Syntax and Semantics}.
\newblock Springer Monographs in Mathematics. Springer International
  Publishing, 2014.

\bibitem[Sch61]{Schensted1961}
Craige Schensted.
\newblock Longest increasing and decreasing subsequences.
\newblock {\em Canad. J. Math.}, 13:179--191, 1961.

\bibitem[Sch77]{Schutzenberger1977d}
Marcel-Paul Sch{\"u}tzenberger.
\newblock La correspondance de {R}obinson.
\newblock In {\em Combinatoire et repr\'esentation du groupe sym\'etrique
  (Actes Table Ronde CNRS, Univ. Louis-Pasteur Strasbourg, Strasbourg, 1976)},
  volume 579 of {\em Lecture Notes in Math.}, pages 61--115. Springer-Verlag,
  1977.

\bibitem[Shn89]{shneerson_axiomaticrank}
Lev~M. Shneerson.
\newblock {O}n the axiomatic rank of varieties generated by a semigroup or
  monoid with one defining relation.
\newblock {\em Semigroup Forum}, 39(1):17--38, December 1989.

\bibitem[Vol01]{volkov_finitebasis}
Mikhail~V. Volkov.
\newblock {T}he finite basis problem for finite semigroups.
\newblock {\em Sci. Math. Jpn}, 53(1):171--199, January 2001.

\end{thebibliography}

\end{document}